\documentclass[a4paper,draft]{amsart}
\usepackage{amsmath, amssymb, latexsym, amsfonts, bbm}
\usepackage[matrix,arrow,curve]{xy}
\usepackage{tikz}
\usetikzlibrary{fit,positioning}

\numberwithin{equation}{section}		%to get same counter
\theoremstyle{plain}
{\bf}{\it}

\newtheorem{theorem}[equation]{Theorem}

\newtheorem{corollary}[equation]{Corollary}

\newtheorem{lemma}[equation]{Lemma}
\theoremstyle{definition}

\newtheorem{definition}[equation]{Definition}

\theoremstyle{remark}
\newtheorem{remark}[equation]{Remark}

\newcommand{\Sp}{{\mathrm{Sp}}}
\newcommand{\SL}{{\mathrm{SL}}}
\newcommand{\Gr}{{\mathrm{Gr}}}
\newcommand{\SGr}{{\mathrm{SGr}}}
\newcommand{\pt}{{\rm pt}}
\newcommand{\Hom}{{\mathrm{Hom}}}
\newcommand{\Z}{{\mathbb{Z}}}

\newcommand{\MSL}{{\mathrm{MSL}}}
\newcommand{\MSp}{{\mathrm{MSp}}}
\newcommand{\KO}{{\mathrm{KO}}}
\newcommand{\GW}{{\mathrm{GW}}}
\newcommand{\Vect}{{\rm Vect}}
\newcommand{\Tc}{{\mathcal{T}}}
\newcommand{\W}{\mathrm{W}}
\newcommand{\A}{\mathbb{A}}
\newcommand{\SSp}{\mathbb{S}}
\newcommand{\PP}{\mathbb{P}}
\newcommand{\rank}{\operatorname{rank}}
\newcommand{\Th}{\operatorname{Th}}
\newcommand{\thc}{\operatorname{th}}
\newcommand{\id}{\operatorname{id}}
\newcommand{\Cone}{\operatorname{Cone}}
\newcommand{\triv}{\mathbf{1}}
\newcommand{\Gm}{{\mathbb{G}_m}}
\newcommand{\SH}{\mathcal{SH}}
\newcommand{\Hp}{\mathcal{H}_\bullet}
\newcommand{\BO}{\mathrm{BO}}
\newcommand{\unit}{e}
\newcommand{\multko}{\dot{\cup}}
\newcommand{\Spec}{\operatorname{Spec}}

\begin{document}

\title[On the relation of special linear cobordism to Witt groups]{On the relation of special linear algebraic cobordism to Witt groups.}

\author{Alexey Ananyevskiy}
\address{Chebyshev Laboratory, 
		 St. Petersburg State University, 
		 14th Line, 29b, 
		 Saint Petersburg, 199178,
		 Russia}
\email{alseang@gmail.com}
\thanks{This research is supported by RFBR grants 13-01-00429, 14-01-31095 and 15-01-03034, by the Chebyshev Laboratory (Department of Mathematics and Mechanics, St. Petersburg State University) under RF Government grant 11.G34.31.0026, by JSC ``Gazprom Neft'' and by ``Dynasty'' foundation.}		 
%\classification{14F42, 19G38, 19E20, 19G12}		 
%\keywords{Witt groups, algebraic cobordism, SL-oriented cohomology, Hopf map}

\begin{abstract}
We reconstruct derived Witt groups via special linear algebraic cobordism. There is a morphism of ring cohomology theories which sends the canonical Thom class in special linear cobordism to the Thom class in the derived Witt groups. We show that for every smooth variety $X$ this morphism induces an isomorphism 
\[
\MSL_{\eta *}^{[\star]}(X)\otimes_{\MSL^{[2\star]}_{\hphantom{[}0}(\pt)}\W^{2\star}(\pt) \to \W^\star(X)[\eta,\eta^{-1}],
\]
where $\eta$ is the stable Hopf map. This result is an analogue of the result by Panin and Walter reconstructing hermitian $K$-theory using symplectic algebraic cobordism.
\end{abstract}

\maketitle

\section{Introduction.}
The main result of this paper relates special linear algebraic cobordism to derived Witt groups. It is a variation of the algebraic version of Conner and Floyd's theorem \cite[Theorem 10.2]{CF} that reconstructs real $K$-theory using symplectic cobordism. The direct algebraic analogue involving symplectic algebraic cobordism and hermitian $K$-theory was obtained by Panin and Walter \cite{PW4}. It claims that for every smooth variety $X$ there exists a canonically defined natural isomorphism
\[
\MSp^{[\star]}_{\hphantom{[}*}(X)\otimes_{\MSp^{[2\star]}_{\hphantom{[}*}(\pt)} \KO^{[2\star]}_{\hphantom{[}0} (\pt) \xrightarrow{\simeq} \KO^{[\star]}_{\hphantom{[}*}(X).
\]
\noindent Here $\KO^{[\star]}_{\hphantom{[}*}(-)$ are Schlichting's hermitian $K$-groups \cite{Sch10a,Sch10b,Sch12} and $\MSp^{[\star]}_{\hphantom{[}*}(-)$ stands for the symplectic algebraic cobordism that is the ring cohomology theory represented in the motivic stable homotopy category $\SH(k)$ by the spectrum $\MSp$ \cite{PW3}. We use a non-standard notation for the double grading of representable cohomology theories which could be rewritten in a more standard way as $A^{[n]}_{\hphantom{[}i}=A^{2n-i,n}$, see Section~2 for the details. Symplectic algebraic cobordism is the universal symplectically oriented cohomology theory \cite[Theorem 13.2]{PW3} and the above isomorphism is induced by the homomorphism $\MSp^{[\star]}_{\hphantom{[}*}(X) \xrightarrow{} \KO^{[\star]}_{\hphantom{[}*}(X)$ arising from the symplectic orientation of hermitian $K$-theory.

In fact, not only is hermitian $K$-theory symplectically oriented, but it is also $\SL$-oriented, thus by the universality of special linear algebraic cobordism \cite[Theorem 5.9]{PW3} there is a natural morphism $\MSL^{[\star]}_{\hphantom{[}*}(X) \xrightarrow{} \KO^{[\star]}_{\hphantom{[}*}(X)$ and one may ask whether we can substitute in the above theorem the special linear cobordism for the symplectic one. In this paper we show that after inverting a certain element $\eta\in \MSL^{[-1]}_{\hphantom{[}-1}(\pt)$ and the corresponding element in $\KO^{[-1]}_{-1}(\pt)$ one indeed obtains an isomorphism. The element $\eta$ arises in the following way. Recall that there is a spectrum $\BO$ representing hermitian $K$-theory \cite{PW2}, i.e. for this spectrum one has natural isomorphisms $\BO^{[n]}_{\hphantom{[} i}(X/U)\cong \KO^{[n]}_{\hphantom{[}i}(X,U)$. Every represented cohomology theory is a module over stable cohomotopy groups $\pi^{[\star]}_{\hphantom{[}*}(\pt)$ in a natural way. Hence for every motivic space $Y$ we have a structure of $\pi^{[\star]}_{\hphantom{[}*}(\pt)$-module on $\MSL^{[\star]}_{\hphantom{[}*}(Y)$ and on $\BO^{[\star]}_{\hphantom{[}*}(Y)$ and can localize these modules at the stable Hopf map $\eta\in\pi^{[-1]}_{\hphantom{[}-1}(\pt)$ corresponding to the $\Sigma_T^\infty$-suspension of $H\colon \A^{2}-\{0\}\to\PP^1, \, H(x,y)=[x:y]$. Theorem~\ref{thm_slCF} claims that for every small pointed motivic space $Y$ the universal morphism $\MSL^{[\star]}_{\hphantom{[}*}(Y) \xrightarrow{} \KO^{[\star]}_{\hphantom{[}*}(Y)$ induces an isomorphism
\[
\MSL^{[\star]}_{\eta *}(Y)\otimes_{\MSL^{[2\star]}_{\hphantom{[}0}(\pt)} \BO^{[2\star]}_{\eta 0} (\pt) \xrightarrow{\simeq} \BO^{[\star]}_{\eta*}(Y).
\]
\noindent See Definition~\ref{def_etaloc} for the precise meaning of $\MSL^{[\star]}_{\eta *}$ and $\BO^{[\star]}_{\eta*}$. 

For every smooth variety $X$ the corresponding pointed motivic space $X_+$ is small and the above formula could be rewritten as
\[
\MSL^{[\star]}_{\eta *}(X)\otimes_{\MSL^{[2\star]}_{\hphantom{[}0}(\pt)} \KO^{[2\star]}_{\eta_S0} (\pt) \xrightarrow{\simeq} \KO^{[\star]}_{\eta_S*}(X),
\]
where $\eta_S\in \KO^{[-1]}_{-1}(\pt)$ is the element corresponding to $\eta \in \BO^{[-1]}_{\hphantom{[}-1}(\pt)$ via the forementioned isomorphism $\BO^{[-1]}_{\hphantom{[}-1}(\pt)\cong \KO^{[-1]}_{-1}(\pt)$. In Theorem~\ref{thm_KOW} we construct natural isomorphisms
\[
\BO^{[\star]}_{\eta *}(X)\cong \KO_{\eta_S*}^{[\star]}(X) \cong \W^{\star}(X)[\eta,\eta^{-1}]
\]
\noindent with an appropriate graded ring structure on the right-hand side. Here $\W^{\star}(X)$ stands for the derived Witt groups defined by Balmer \cite{Bal99}, one can find a comprehensive survey including definitions and applications of these groups in \cite{Bal05}. Combining the above observations in Corollary~\ref{cor_slCFW} we obtain for a smooth variety $X$ a natural isomorphism
\[
\MSL^{[\star]}_{\eta *}(X)\otimes_{\MSL^{[2\star]}_{\hphantom{[}0}(\pt)} \W^{2\star} (\pt) \xrightarrow{\simeq} \W^{\star}(X)[\eta,\eta^{-1}].
\]

The paper is organized as follows. In section 2 we recall the general context of unstable $\Hp(k)$ and stable $\SH(k)$ motivic homotopy categories introduced by Morel and Voevodsky \cite{MV,V}. Then we do some preliminary calculations with stable cohomotopy groups $\pi^{[\star]}_{\hphantom{[}*}$ staying mainly in the unstable homotopy category $\Hp(k)$. In section 4 we deal with special linear and symplectic orientations and recall the universality theorems for the algebraic cobordism $\MSL$ and $\MSp$. In the next two sections we deal with hermitian $K$-theory and the spectrum $\BO$, in particular we show that $\BO^{[\star]}_{\eta *}(X)$ is isomorphic to the Laurent polynomial ring over the derived Witt groups $\W^{\star}(X)$. 

The last section is devoted to the main theorem relating $\MSL$-cobordism to the derived Witt groups. Using Panin and Walter's result one can easily show that the examined homomorphism is surjective and construct a section. The main issue is to show that the section maps the Thom classes of special linear bundles to the corresponding Thom classes. It is an easy observation that the claim holds for the Thom classes of symplectic bundles and, thanks to the theory of characteristic classes developed in \cite{An15}, the general case follows from this observation. Recall that for an $\SL$-oriented cohomology theory one has a good theory of characteristic classes only after inverting the stable Hopf map $\eta$, for the details see loc.cit. In this case one has analogues of the projective bundle theorem and the splitting principle. Roughly speaking, the splitting principle claims that working with $\SL$-oriented cohomology theories with inverted stable Hopf map one may think that every special linear bundle of even rank is a direct sum of special linear bundles of rank two. But a special linear bundle of rank two is a symplectic bundle in a natural way, thus every special linear bundle of even rank is, in a certain sense, a symplectic bundle. See \cite[Theorem~7]{An15} for the precise statement.

\textbf{Acknowledgments.}
The author wishes to express his sincere gratitude to I.~Panin for the numerous discussions on the subject of this paper and to the anonymous referee, whose thorough comments greatly improved the presentation.

\section{Preliminaries on $\SH(k)$ and ring cohomology theories.}

Throughout this paper $k$ is a field of characteristic different from $2$. 

Let $Sm/k$ be the category of smooth varieties over $k$. A motivic space over $k$ is a simplicial presheaf on $Sm/k$. Each variety $X\in Sm/k$ defines a motivic space $\Hom_{Sm/k}(-,X)$ constant in the simplicial direction. We write $\pt$ for $\Spec k$ regarded as a motivic space. Inverting the weak motivic equivalences in the category of pointed motivic spaces gives the pointed motivic unstable homotopy category $\Hp (k)$. The smash-products of pointed simplicial presheaves induces a symmetric monoidal structure on $\Hp(k)$ with the unit given by $\pt_+$.

Let $T=\A^1/(\A^1-\{0\})$ be the Morel-Voevodsky object. A $T$-spectrum $M$ is a sequence of pointed motivic spaces $(M_0,M_1,M_2,\dots)$ equipped with maps $\sigma_n  \colon T\wedge M_n\to M_{n+1}$. A map of $T$-spectra is a sequence of maps of pointed motivic spaces which is compatible with the structure maps. Inverting the stable motivic weak equivalences as in \cite{Jar} gives the motivic stable homotopy category $\SH(k)$. See loc. cit. for the discussion of symmetric monoidal structure on $\SH(k)$ based on motivic symmetric spectra.

A pointed motivic space $Y$ gives rise to a suspension $T$-spectrum 
\[
\Sigma^{\infty}_T Y=(Y,T\wedge Y, T\wedge T \wedge Y,\hdots).
\]
Put $\SSp=\Sigma^{\infty}_T (\pt_+)$ for the spherical spectrum. This spectrum is the unit for the monoidal structure on $\SH(k)$ constructed in \cite{Jar} and the suspension functor
\[
\Sigma^{\infty}_T\colon \Hp(k)\to \SH (k)
\]
is a strict symmetric monoidal functor. 

Recall that there are two spheres in $\Hp(k)$, the simplicial one $S^{1,0}= \Delta^1/\partial(\Delta^1)$ and $S^{1,1}=(\Gm,1)$. For integers $p,q\ge 0$ we write $S^{p+q,q}$ for $(S^{1,1})^{\wedge q}\wedge (S^{1,0})^{\wedge p}$ and $\Sigma^{p+q,q}$ for the suspension functor $-\wedge S^{p+q,q}$. This functor becomes invertible in the stable homotopy category $\SH(k)$, so we extend the notation to arbitrary integers $p,q$ in an obvious way. 

Any $T$-spectrum $A$ defines a bigraded cohomology theory on the category of pointed motivic spaces. Namely, for a pointed motivic space $(Y,y)$ one sets
\[
A^{[n]}_{\hphantom{[}i}(Y,y)=\Hom_{\SH(k)}(\Sigma^\infty_T (Y,y),\Sigma^{2n-i,n}A)
\]
and $A^{[\star]}_{\hphantom{[}*}(Y,y)=\bigoplus\limits_{n,i}A^{[n]}_{\hphantom{[} i}(Y,y)$. If $p-q,q\ge 0$ one has a canonical suspension isomorphism $A^{[n]}_{\hphantom{[} i}(Y,y)\cong A^{[n+q]}_{i-p+2q}(\Sigma^{p,q}(Y,y))$ induced by the permutation isomorphism $S^{2n-i,n}\wedge S^{p,q}\cong S^{2n-i+p,n+q}$. 
In the motivic homotopy category there is a canonical isomorphism $T\cong S^{2,1}$, we put 
\[
\Sigma_T\colon A^{[\star]}_{\hphantom{[}*}(Y,y)\xrightarrow{\simeq} A^{[\star+1]}_{\hphantom{[}*}((Y,y)\wedge T)
\]
for the corresponding suspension isomorphism. See Definition~\ref{def_susp} for the details. 

For a variety $X$ we put $A^{[n]}_{\hphantom{[}i}(X)=A^{[n]}_{\hphantom{[}i}(X_+,+)$ for the externally pointed motivic space $(X_+,+)$ associated to $X$. Groups $A^{[\star]}_{\hphantom{[}*}(X)$ are defined accordingly. Put $\pi^{[\star]}_{\hphantom{[}*}(X)=\SSp^{[\star]}_{\hphantom{[}*}(X)$ to be the stable cohomotopy groups of $X$. Given a closed embedding $i\colon Z\to X$ of varieties we write $\Th(i)$ for $X/(X-Z)$. For a particular case of the zero section $z\colon X\to E$ of a vector bundle we put 
\[
\Th(E)=\Th(z)=E/(E-X)
\]
and refer to it as \textit{Thom space} of $E$. Recall that for a $T$-spectrum $A$ a closed embedding $i\colon Z\to X$ of varieties gives rise to the long exact \textit{localization sequence}
\[
\hdots \xrightarrow{\partial} A^{[\star]}_{\hphantom{[}*}(\Th(i)) \xrightarrow{p^A} A^{[\star]}_{\hphantom{[}*}(X) \xrightarrow{j^A} A^{[\star]}_{\hphantom{[}*}(X-Z) \xrightarrow{\partial} A^{[\star]}_{*-1}(\Th(i)) \xrightarrow{} \hdots
\]
Here $j\colon X-Z\to X$ and $p\colon X\to \Th(i)$ are the obvious embedding and projection morphisms respectively. The localization sequence is the long exact sequence associated to the cofibration $j\colon X-Z\to X$.

A commutative ring $T$-spectrum is a commutative monoid 
\[
(A,\, m_A\colon A\wedge A\to A,\, \unit_A\colon \SSp\to A)
\]
in $\SH(k)$. The cohomology theory defined by such spectrum is a ring cohomology theory satisfying a certain bigraded commutativity condition described by Morel (see below). Recall some facts related to the multiplicative structure which are parallel to the classical topological ones studied in \cite[III.9]{Ad74}.

(1) \textit{Cross-product:} let $Y_1$ and $Y_2$ be pointed motivic spaces. Then there exists a functorial bilinear pairing
\[
\times\colon A^{[n]}_{\hphantom{[}i}(Y_1)\times A^{[m]}_{\hphantom{[}j}(Y_2)\to A^{[n+m]}_{\hphantom{[}i+j}(Y_1\wedge Y_2)
\]
given by $ a\times b= (m_A \wedge \sigma) \circ (\id_A \wedge \tau_{S^{2n-i,n},A}\wedge \id_{S^{2m-j,m}}) \circ (a\wedge b)$,
\[
\xymatrix @C=6pc{
Y_1 \wedge Y_2 \ar[r]^(0.4){a\wedge b} & A\wedge S^{2n-i,n}\wedge A\wedge S^{2m-j,m} \ar[d]^{\labelstyle \id_A \wedge \tau_{S^{2n-i,n},A}\wedge \id_{S^{2m-j,m}}} &  \\
& A\wedge A \wedge S^{2n-i,n} \wedge S^{2m-j,m} \ar[r]^{m_A \wedge \sigma} & A\wedge S^{2(n+m)-(i+j),n+m},
}
\]
where 
\[
\tau_{S^{2n-i,n},A}\colon S^{2n-i,n}\wedge A\xrightarrow{\simeq} A\wedge S^{2n-i,n},\quad \sigma\colon S^{2n-i,n} \wedge S^{2m-j,m}\xrightarrow{\simeq} S^{2(n+m)-(i+j),n+m}
\]
are canonical permutation isomorphisms.

(2) \textit{Cup-product:} for a pointed motivic space $Y$ there is a functorial graded ring structure
\[
\cup\colon A^{[\star]}_{\hphantom{[}*}(Y)\times A^{[\star]}_{\hphantom{[}*}(Y)\to A^{[\star]}_{\hphantom{[}*}(Y)
\]
given by $a\cup b= \Delta^A(a\times b)$, where $\Delta \colon Y\to Y\wedge Y$ is the diagonal morphism.
Moreover, let $i_1\colon Z_1\to X$ and $i_2\colon Z_2\to X$ be closed embeddings of varieties and put $i_{12}\colon Z_1\cap Z_2\to X$. Then there is a functorial, bilinear and associative cup-product
\[
\cup \colon A^{[\star]}_{\hphantom{[}*}(\Th(i_1))\times A^{[\star]}_{\hphantom{[}*}(\Th(i_2)) \to A^{[\star]}_{\hphantom{[}*}(\Th(i_{12}))
\]
given by $a\cup b = \widetilde{\Delta}^A(a\times b)$, where $\widetilde{\Delta} \colon \Th(i_{12}) \to \Th(i_{1})\wedge \Th(i_{2})$ is induced by the diagonal embedding $\Delta\colon X\to X\times X$.
In particular, setting $Z_1=X$ we have $\Th(i_1)=X_+$ and therefore obtain an $A^{[\star]}_{\hphantom{[}*}(X)$-module structure on $A^{[\star]}_{\hphantom{[}*}(\Th(i_2))$.
We will usually omit $\cup$ from the notation.

(3) \textit{Module structure over stable cohomotopy groups:} the unit morphism $\unit_A\colon \SSp\to A$ induces a homomorphism of graded rings $\pi^{[\star]}_{\hphantom{[}*}(\pt)\to A^{[\star]}_{\hphantom{[}*}(\pt)$. For every pointed motivic space $Y$ this homomorphism together with the cross-product define a $\pi^{[\star]}_{\hphantom{[}*}(\pt)$-bimodule structure on $A^{[\star]}_{\hphantom{[}*}(Y)$. For a variety $X$ there is a canonical morphism $X\to \pt$ inducing a homomorphism $\pi^{[\star]}_{\hphantom{[}*}(\pt)\to \pi^{[\star]}_{\hphantom{[}*}(X)$. We equip $A^{[\star]}_{\hphantom{[}*}(X)$ with the structure of a graded unital $\pi^{[\star]}_{\hphantom{[}*}(\pt)$-algebra using the composition $\pi^{[\star]}_{\hphantom{[}*}(\pt)\to \pi^{[\star]}_{\hphantom{[}*}(X)\to A^{[\star]}_{\hphantom{[}*}(X)$ and cup-product. The unit of $A^{[\star]}_{\hphantom{[}*}(X)$ is usually denoted by $1_X=1_X^A$.

(4) \textit{Graded $\epsilon$-commutativity} \cite[Lemma 6.1.1]{Mor04}: let $\epsilon\in \pi^{[0]}_{\hphantom{[}0}(\pt)$ be the element corresponding under the $T$-suspension isomorphism to the morphism $T\to T, x\mapsto -x$. Then for every pointed motivic space $Y$ and $a\in A^{[n]}_{\hphantom{[}i}(Y)$, $b\in A^{[m]}_{j}(Y)$ we have
\[
ab =(-1)^{ij}\epsilon^{nm}ba.
\]
Recall that $\epsilon^2=1_{\pt}$.

\section{Motivic spheres.}
In this section we recall certain canonical isomorphisms in the homotopy category and the definition of the stable Hopf map. Then we carry out a number of computations in the homotopy category in order to present in a convenient way the connecting homomorphism in the localization sequence for the embedding $\{0\}\subset \A^1$. The goal of this section is to show that one can obtain the stable Hopf map applying this connecting homomorphism to a certain natural element of the cohomotopy groups. Throughout this section we will write $\Gm,\A^1,\PP^1$ for the corresponding motivic spaces pointed by $1$ and $\Gm_{+},\A^1_+,\PP^1_+$ for the motivic spaces pointed externally. 

In order to write down the canonical isomorphisms for the different models of the motivic spheres one needs the cone construction.

\begin{definition}
Let $i\colon Y\to Y'$ be a morphism of pointed motivic spaces. The space $\Cone(i)$ defined via the cocartesian square
\[
\xymatrix{
Y \ar[r]^{i}\ar[d]^{in_1} & Y' \ar[d] \\
Y\wedge \Delta^1 \ar[r] & \Cone(i)
}\quad
\]
is called the \textit{cone of the morphism $i$}. Here $in_1\colon Y\cong Y\wedge \pt_+ \to Y\wedge \Delta^1$ is induced by the embedding $\pt_+=\partial\Delta^1\subset \Delta^1$.
\end{definition}

The following isomorphisms are well-known, see \cite[Lemma 2.15, Example 2.20]{MV}.

\begin{definition}
\label{def_rho}
Set $\rho=\rho_2\circ\rho_1^{-1}\colon T\xrightarrow{\simeq} \Gm\wedge S^1_s=S^{2,1}$ for the canonical isomorphism in the homotopy category defined via
\[
T\xleftarrow{\rho_1} \Cone(i) \xrightarrow{\rho_2} \Gm\wedge S^1_s
\]
where $i$ stands for the natural embedding $\Gm\to \A^1$ and the isomorphisms $\rho_1$ and $\rho_2$ are induced by the maps $\Delta^1\to \pt$ and $\A^1\to \pt$ respectively.
\end{definition}

\begin{definition}
Put $\sigma=\sigma_2^{-1}\circ\sigma_1\colon (\A^2-\{0\},(1,1))\xrightarrow{\simeq} \Gm\wedge T$ for the canonical isomorphism in the homotopy category. It is defined via
\[
(\A^2-\{0\},(1,1)) \xrightarrow{\sigma_1} (\A^2-\{0\})/((\A^1\times \Gm)\cup(\{1\}\times \A^1)) \xleftarrow{\sigma_2} \Gm\wedge T
\]
where $\sigma_1$ is induced by the identity map on $\A^2-\{0\}$ and $\sigma_2$ is induced by the natural embedding $\Gm\times \A^1\subset \A^2-\{0\}$. Recall that $\sigma_1$ is an isomorphism since the space $(\A^1\times \Gm)\cup(\{1\}\times \A^1)$ is $\A^1$-contractible, while $\sigma_2$ is induced by the excision isomorphism $\Gm_+\wedge T\cong (\A^2-\{0\})/(\A^1\times \Gm),$ so it is an isomorphism as well.
\end{definition}

\begin{definition}
\label{def_susp}
Let $A$ be a $T$-spectrum and $Y$ be a pointed motivic space. Identifying $T\cong S^{2,1}$ via $\rho$ we denote 
\[
\Sigma_T=(\id_Y\wedge \rho)^A\circ\Sigma^{2,1}\colon A^{[n]}_{\hphantom{[}i}(Y)\xrightarrow{\simeq}A^{[n+1]}_{\hphantom{[}i}(Y\wedge T)
\] 
the \textit{$T$-suspension isomorphism}.
\end{definition}

\begin{definition}
The \textit{Hopf map} is the morphism of the varieties
\[
H\colon \A^2-\{0\} \to \PP^1
\]
defined via $H(x,y)=[x,y]$. Pointing $\A^2-\{0\}$ by $(1,1)$ and $\PP^1$ by $[1:1]$ and taking the suspension spectra we obtain the corresponding morphism
\[
\Sigma^\infty_T H \in \Hom_{\SH (k)}(\Sigma^\infty_T(\A^2-\{0\},(1,1)),\Sigma^\infty_T \PP^1).
\]
In order to interpret this morphism as an element of $\pi^{[-1]}_{\hphantom{[}-1}(\pt)$ we introduce the following isomorphism. Let $\vartheta=\rho \vartheta_2^{-1}\vartheta_1\in \Hp(k)$ be the composition 
\[
\vartheta\colon \PP^1\xrightarrow{\vartheta_1}\PP^1/\A^1\xleftarrow{\vartheta_2} T\xrightarrow{\rho} S^{2,1}.
\]
Here $\vartheta_1$ is induced by the identity map on $\PP^1$ and $\vartheta_2$ is the excision isomorphism given by $\vartheta_2(x)=[x:1]$.
\textit{Stable Hopf map} is the unique element
\[
\eta \in \pi^{[-1]}_{\hphantom{[}-1}(\pt) \cong \pi^{[1]}_{\hphantom{[}0}(\Gm\wedge T)  \cong \pi^{[1]}_{\hphantom{[}0}(\A^2-\{0\},(1,1)) \cong \Hom_{\SH (k)}(\Sigma^\infty_T(\A^2-\{0\},(1,1)),\Sigma^\infty_T \PP^1)
\]
subject to the relation
\[
\sigma^\pi(\Sigma_T\Sigma^{1,1}\eta)=\Sigma^\infty_T (\vartheta \circ H).
\]
In other words, stable Hopf map $\eta$ is the element of $\pi^{[-1]}_{\hphantom{[}-1}(\pt)$ corresponding to $\Sigma_T^\infty H$ via the canonical isomorphisms between motivic spheres and suspension isomorphisms. For a commutative ring $T$-spectrum $A$ we denote by the same letter $\eta$ the image of the stable Hopf map in $A^{[-1]}_{\hphantom{[}-1}(\pt)$ under the homomorphism $\pi^{[-1]}_{\hphantom{[}-1}(\pt)\to A^{[-1]}_{\hphantom{[}-1}(\pt)$ induced by the unit morphism $\SSp\xrightarrow{\unit_A} A$.
\end{definition}

\begin{definition}
\label{def_symb}
Let $A\in \SH(k)$ be a commutative ring $T$-spectrum. Consider a variety $X$ and an invertible function $f\in k[X]^*$ on $X$. This function defines an automorphism 
\[
f_T\colon X_+\wedge T\to X_+\wedge T
\] 
via $f_T(x,t)=(x,f(x)t)$. We call the element
\[
\langle f\rangle_A=\Sigma_T^{-1}(f_T^A(\Sigma_T 1_X)) \in A^{[0]}_{\hphantom{[}0}(X)
\]
the \textit{symbol associated to $f$ in $A$}.
\end{definition}

\begin{remark}
By the very definition we have $\langle -1\rangle_\pi=\epsilon$.
\end{remark}

Consider the closed embedding $\{0\}\to \A^1$ and the localization sequence
\[
A^{[0]}_{\hphantom{[}0}(\A^1_+)\to A^{[0]}_{\hphantom{[}0}(\Gm_+)\xrightarrow{\partial} A^{[0]}_{-1}(T).
\]
Our goal is to prove the following theorem.
\begin{theorem}
\label{thm_etacoh}
Let $A\in \SH(k)$ be a commutative ring $T$-spectrum and let 
\[
\partial \colon A^{[0]}_{\hphantom{[}0}(\Gm_+)\to A^{[0]}_{-1}(T)
\]
be the connecting homomorphism in the long exact localization sequence for the embedding $\{0\}\subset \A^1$. Then for the coordinate function $t\in k[\Gm]^*$ we have
\[
\partial(\langle-t^{-1}\rangle_A)=\Sigma_T\eta.
\]
\end{theorem}

Using the morphism of cohomology theories $\pi^{[\star]}_{\hphantom{[}*}(-)\to A^{[\star]}_{\hphantom{[}*}(-)$ induced by the unit morphism $\unit_A\colon \SSp\to A$ one easily sees that it is sufficient to treat the case of $A=\SSp$. For the remaining part of the section we assume that $A=\SSp$ and work with the stable cohomotopy groups $\pi^{[\star]}_{\hphantom{[}*}(-)=\SSp^{[\star]}_{\hphantom{[}*}(-)$.

In order to prove the above theorem we need the next lemmas describing $\partial$ in a convenient way. Set $\delta=\Sigma^{-1,0}(\rho^{\pi})^{-1}\partial$,
\[
\xymatrix{
\pi^{[n]}_{\hphantom{[}i}(\Gm_+) \ar[d]_{\delta} \ar[r]^{\partial} & \pi^{[n]}_{i-1}(T) \\
\pi^{[n]}_{\hphantom{[}i}(\Gm) \ar[r]^(0.4){\Sigma^{1,0}}_(0.4){\simeq} & \pi^{[n]}_{i-1}(\Gm\wedge S^{1}_s). \ar[u]_{\rho^\pi}^{\simeq}
}
\]

\begin{lemma}
For the natural map $r\colon \Gm_{+}\to \Gm$ we have $\delta r^{\pi}=\id$.
\label{lem_cohspit}
\end{lemma}
\begin{proof}
The statement of the lemma is equivalent to $\partial r^\pi=\rho^\pi\Sigma^{1,0}$. Denote $i\colon \Gm\to \A^1$ and $i_+\colon \Gm_+\to \A^1_+$ the natural embeddings and let $j_1\colon \A^1_+\to \Cone(i_+)$ and $j_2\colon \Cone(i_+)\to \Cone(j_1)$ be the natural maps for the cone construction. 

Consider the following diagram.
\[
\xymatrix{
\Gm\wedge S^{1}_s & \Cone(i) \ar[r]^(0.6){\rho_1}_(0.6){\simeq} \ar[l]_(0.4){\rho_2}^(0.4){\simeq}   & T \\
& \Cone(i_+) \ar[d]^{j_2}\ar[ld]_{u}\ar[ur]^{\simeq}_{\psi_1}\ar[u]_{v}^{\simeq}
%\ar[ul]_{\simeq}^{\psi_2} 
& \\
\Gm_+\wedge S^1_s  \ar[uu]^{r\wedge \id} & \Cone(j_1) \ar[l]_(0.4){w}^(0.4){\simeq} & 
}
\]
Here 
$\psi_1$ and $w$ are induced by $\Delta^1\to \pt$, $v$ is induced by $\{+\}\to \{1\}$ and $u$ is induced by $\A^1_+\to \pt$. One can check that this diagram is commutative in the homotopy category. By the very definition we have 
\[
\partial r^\pi=(wj_2\psi_1^{-1})^\pi\Sigma^{1,0}r^\pi=(wj_2\psi_1^{-1})^\pi(r\wedge \id)^\pi\Sigma^{1,0}=((r\wedge \id)wj_2\psi_1^{-1})^\pi\Sigma^{1,0},
\]
thus it is sufficient to show
\[
(r\wedge \id)wj_2\psi_1^{-1}=\rho_2\rho_1^{-1}=\rho
\]
and this follows from the commutativity of the above diagram.
\end{proof}

Suspending $\delta$ with $\Sigma_T$ and shifting the indices we obtain a homomorphism $\delta_T=\Sigma_T\delta\Sigma_T^{-1}$,
\[
\xymatrix{
\pi^{[n]}_{\hphantom{[}i}(\Gm_+\wedge T) \ar[r]^{\delta_T}  &  \pi^{[n]}_{\hphantom{[}i}(\Gm\wedge T) \\
\pi^{[n-1]}_{\hphantom{[}i}(\Gm_+) \ar[r]^{\delta} \ar[u]^{\Sigma_T}_{\simeq} & \pi^{[n-1]}_{\hphantom{[}i}(\Gm) \ar[u]^{\Sigma_T}_{\simeq}.
}
\]
Put $Y=(\A^2-\{0\})/((\A^1\times \Gm)\cup(\{1\}\times \A^1))$ and consider the following diagram
\[
\xymatrix{
\pi^{[n]}_{\hphantom{[}i}((\A^2-\{0\})/(\A^1\times \Gm)) \ar[rr]^(0.6){\phi^\pi}_(0.6){\simeq} \ar[d]_{\psi_1^\pi} & & \pi^{[n]}_{\hphantom{[}i}(\Gm_+\wedge T) \ar[d]_{\delta_T}\\
\pi^{[n]}_{\hphantom{[}i}(\A^2-\{0\},(1,1)) & \pi^{[n]}_{\hphantom{[}i}(Y) \ar[l]_(0.35){\sigma_1^\pi}^(0.35){\simeq} \ar[ul]_(0.4){\psi_2^\pi}  \ar[r]^(0.45){\sigma_2^\pi}_(0.45){\simeq} & \pi^{[n]}_{\hphantom{[}i}(\Gm\wedge T) \ar@<-0.5pc>[u]_{(r\wedge \id)^\pi}
}
\]
where $\psi_1$ and $\psi_2$ are induced by the identity map on $\A^2-\{0\}$, $\phi^\pi$ is the excision isomorphism induced by the embedding $\Gm\times\A^1\subset \A^2-\{0\}$ and $r\colon \Gm_+\to \Gm$ is the obvious map.

\begin{lemma}
\label{lem_concoh}
In the above diagram we have
\[
\delta_T=\sigma_2^\pi(\sigma_1^\pi)^{-1}\psi_1^\pi(\phi^\pi)^{-1}.
\]
\end{lemma}
\begin{proof}
Put $\delta_T'=\sigma_2^\pi(\sigma_1^\pi)^{-1}\psi_1^\pi(\phi^\pi)^{-1}$. By Lemma~\ref{lem_cohspit} we have
\[
\delta_T(r\wedge \id)^\pi=\Sigma_T\delta\Sigma_T^{-1}(r\wedge \id)^\pi=\Sigma_T\delta r^\pi\Sigma_T^{-1}=\id.
\]
On the other hand, commutativity of the diagram 
\[
\xymatrix{
(\A^2-\{0\})/(\A^1\times \Gm) \ar[dr]^(0.6){\psi_2} & & \Gm_+\wedge T \ar[ll]_(0.4){\phi}^(0.4){\simeq} \ar[d]^{r\wedge \id}\\
(\A^2-\{0\},(1,1)) \ar[u]^{\psi_1} \ar[r]^(0.65){\sigma_1}_(0.65){\simeq} & Y  & \Gm\wedge T  \ar[l]_(0.55){\sigma_2}^(0.55){\simeq} 
}
\]
implies $(r\wedge \id)\phi^{-1}\psi_1\sigma_1^{-1}\sigma_2 =\id$ and
\[
\delta_T'(r\wedge \id)^\pi=\sigma_2^\pi(\sigma_1^\pi)^{-1}\psi_1^\pi(\phi^\pi)^{-1}(r\wedge \id)^\pi=\id.
\]
At last, set $s_1\colon \Gm_+\wedge T\to T$  and $s_2\colon (\A^2-\{0\})/(\A^1\times \Gm) \to T$ for the projections $(x,y)\mapsto y$ and note that $s_2\phi=s_1$. Recall that the localization sequence for the embedding $\{0\} \to \A^1$ admits a splitting yielding 
\begin{multline*}
\ker \delta_T=\Sigma_T \ker\partial=\Sigma_T(\operatorname{Im}(\pi^{[n]}_{\hphantom{[}i}(\A^1_+)\to \pi^{[n]}_{\hphantom{[}i}(\Gm_+)))=\Sigma_T(\operatorname{Im}(\pi^{[n]}_{\hphantom{[}i}(\pt_+)\to \pi^{[n]}_{\hphantom{[}i}(\Gm_+)))=\\=s_1^\pi(\pi^{[n]}_{\hphantom{[}i}(T)).
\end{multline*}
We have $s_2\psi_1=0$ in the homotopy category, so $\delta_T'\phi^\pi s_2^\pi=0$. Thus we obtain
\[
\ker \delta_T=s_1^\pi(\pi^{[n]}_{\hphantom{[}i}(T))=\phi^\pi s_2^\pi(\pi^{[n]}_{\hphantom{[}i}(T))\subset \ker \delta_T'.
\]

To sum up, there are two homomorphisms $\delta_T,\delta_T'$ with the same splitting and $\ker \delta_T\subset \ker \delta_T'$. Then for every $a\in\pi^{[n]}_{\hphantom{[}i}(\Gm_+\wedge T)$ we have
\[
\delta_T(a-(r\wedge \id)^{\pi}\delta_T(a))=0=\delta_T'(a-(r\wedge \id)^{\pi}\delta_T(a))=\delta_T'(a)-\delta_T(a),
\]
yielding $\delta_T=\delta_T'$.
\end{proof}

\begin{proof}[Proof of Theorem~\ref{thm_etacoh}]
Consider the morphism $m\colon \Gm_+\wedge T\to \Gm_+\wedge T$ given by $m(t,x)=(t,-x/t)$. Unraveling the definitions, we need to show
\[
\partial\Sigma_T^{-1}(m^\pi\Sigma_T(1))=\Sigma_T\Sigma^{-1,-1}\Sigma_T^{-1}(\sigma^{\pi})^{-1}(\Sigma^{\infty}_T (\vartheta \circ H)).
\]
The element $m^\pi\Sigma_T1\in\pi^{[1]}_{\hphantom{[}0}(\Gm_+\wedge T)$ could be represented by the $\Sigma^\infty_T$-suspension of the composition
\[
\Gm_+\wedge T\xrightarrow{\widetilde{H}_1} T\xrightarrow{\rho} S^{2,1},
\]
with $\widetilde{H}_1$ given by $\widetilde{H}_1(t,x)=-x/t$. We can rewrite this composition as $\rho\widetilde{H}_1=\rho\vartheta_2^{-1}H_1$ with $H_1\colon \Gm_+\wedge T\to \PP^1/\A^1$ given by $H_1(t,x)=[x:-t]$. Hence on the left-hand side we have $\partial\Sigma_T^{-1}(\Sigma^{\infty}_T(\rho\circ\vartheta_2^{-1}\circ H_1))$.

Unraveling the definition of $\delta_T$, by Lemma~\ref{lem_concoh} we obtain
\[
\Sigma_T\Sigma^{-1,0}(\rho^\pi)^{-1}\partial\Sigma_T^{-1}=(\sigma^\pi)^{-1}\psi_1^{\pi}(\phi^\pi)^{-1},\quad \partial\Sigma_T^{-1}=\rho^\pi\Sigma^{1,0}\Sigma_T^{-1}(\sigma^\pi)^{-1}\psi_1^{\pi}(\phi^\pi)^{-1}.
\]
Thus on the left-hand side we have
\[
\rho^\pi\Sigma^{1,0}\Sigma_T^{-1}(\sigma^\pi)^{-1}\psi_1^{\pi}(\phi^\pi)^{-1}(\Sigma^{\infty}_T(\rho\circ\vartheta_2^{-1}\circ H_1)).
\]
Consider the following commutative diagram.
\[
\xymatrix{
\PP^1 \ar[r]^{\vartheta_1}_{\simeq} & \PP^1/\A^1 \ar[r]^{\tau}_{=}& \PP^1/\A^1 \\
(\A^2-\{0\},(1,1)) \ar[r]^(0.45){\psi_1} \ar[u]^{H} \ar[ur]^{H_3} & (\A^2-\{0\})/(\A^1\times \Gm)  \ar[ur]^{H_2}  &  \Gm_+\wedge T \ar[l]_(0.32){\phi}^(0.32){\simeq}  \ar[u]_{H_1}
}
\]
Here $H_1$ and $H_2$ are given by $(x,y)\mapsto [y:-x]$, $H$ and $H_3$ are given by $(x,y)\mapsto [x:y]$, $\PP^1\cong \PP^1/\A^1$ is the canonical isomorphism and $\tau$ is given by $[x,y]\mapsto [y,-x]$. Recall that in the homotopy category $\tau=\id$, thus we have
\[
\psi_1^{\pi}(\phi^\pi)^{-1}(\Sigma^{\infty}_T(\rho\vartheta_2^{-1}H_1))=\Sigma^{\infty}_T(\rho\vartheta_2^{-1}H_1 \phi^{-1} \psi_1)=\Sigma^\infty_T (\rho\vartheta_2^{-1}\vartheta_1H)=\Sigma_T^\infty (\vartheta H).
\]

Summing up the above considerations, it is left to show that
\[
\rho^\pi\Sigma^{1,0}\Sigma_T^{-1}(\sigma^\pi)^{-1}(\Sigma^\infty_T (\vartheta H))=\Sigma_T\Sigma^{-1,-1}\Sigma_T^{-1}(\sigma^{\pi})^{-1}(\Sigma^{\infty}_T(\vartheta H)),
\]
which follows from the definition of $\Sigma_T$.

\end{proof}

\section{Special linear and symplectic orientations.}
In this section we briefly recall the definitions of different types of orientations. A detailed exposition can be found in \cite{PW1,PW2}. Mostly for the case of uniformity, only the case of a representable cohomology theory is treated. Fix for this section a commutative ring $T$-spectrum $A\in\SH(k)$ representing a ring cohomology theory $A^{[\star]}_{\hphantom{[}*}(-)$.

Roughly speaking, an orientation on a cohomology theory is a rule that fixes for every vector bundle $E$ over every smooth variety $X$ a natural Thom class $\thc (E)\in A^{[\star]}_{\hphantom{[}*}(\Th(E))$ such that
\[
-\cup \thc (E)\colon A^{[\star]}_{\hphantom{[}*}(X)\to A^{[\star+ \rank E]}_{\hphantom{[}*}(\Th(E))
\]
is an isomorphism \cite{PS}. Particular types of orientation fix such classes only for vector bundles with additional structure, for example for symplectic or for special linear ones. Note that these classes usually do depend on the additional structure, i.e. for the same vector bundle with different symplectic forms one could have different Thom classes.

\begin{definition}
A \textit{special linear bundle} over a variety $X$ is a pair $(E, \lambda)$ with $E\to X$ a vector bundle and $\lambda \colon \det E\xrightarrow{\simeq}\triv_X$ an isomorphism of line bundles. Here $\triv_X$ denotes the trivial line bundle over $X$.
\end{definition}

\begin{definition}
\label{def_SLorient}
A \textit{(normalized) special linear orientation} on a ring cohomology theory $A^{[\star]}_{\hphantom{[}*}(-)$ is a rule which assigns to every special linear bundle $(E,\lambda)$ of rank $n$ over a smooth variety $X$ a class $\thc(E,\lambda)\in A^{[n]}_{\hphantom{[}0}(\Th(E))$ satisfying the following conditions \cite[Definition~5.1]{PW3}:
\begin{enumerate}
\item
for an isomorphism $f\colon (E,\lambda)\xrightarrow{\simeq} (E',\lambda')$ (i.e. an isomorphism $f\colon E\to E'$ satisfying $\lambda=\det f \circ \lambda'$) we have $\thc(E,\lambda)=f^A\thc(E',\lambda')$;
\item
for a morphism of smooth varieties $r\colon X'\to X$ we have $r^A\thc(E,\lambda)=\thc(r^*(E,\lambda))$;
\item
homomorphisms $-\cup \thc(E,\lambda)\colon A^{[\star]}_{\hphantom{[}*}(X)\to A^{[\star+n]}_{\hphantom{[}*}(\Th(E))$
are isomorphisms;
\item
we have
\[
\thc(E_1\oplus E_2, \lambda_1\otimes \lambda_2)=q_1^A\thc(E_1,\lambda_1)\cup q_2^A\thc(E_2,\lambda_2),
\]
where $q_1,q_2$ are projections from $E_1\oplus E_2$ onto its summands;
\item
for the zero bundle $\mathbf{0}\to \pt$ over the point we have $\thc(\mathbf{0})=1_\pt\in A^{[0]}_{\hphantom{[}0}(\pt)$;
\item
for the trivial special linear line bundle $\triv_\pt$ over the point we have
\[
\thc(\triv_\pt,\id)=\Sigma_T 1_\pt\in A^{[1]}_{\hphantom{[}0}(T).
\]
\end{enumerate}
The class $\thc(E,\lambda)$ is the \textit{Thom class} of the special linear bundle. We call a ring cohomology theory with a normalized special linear orientation an \textit{$\SL$-oriented cohomology theory}.
\end{definition}

One can give an analogous definition of the symplectic orientation on a cohomology theory. 

\begin{definition}
A \textit{(normalized) symplectic orientation}  on a ring cohomology theory $A^{[\star]}_{\hphantom{[}*}(-)$ is a rule which assigns to every symplectic bundle $(E,\phi)$ of rank $n$ over a smooth variety $X$ a class $\thc(E,\phi)\in A^{[n]}_{\hphantom{[}0}(\Th(E))$ satisfying the following conditions \cite[Definition~14.2]{PW1}:
\begin{enumerate}
\item
for an isomorphism $f\colon (E,\phi)\xrightarrow{\simeq} (E',\phi')$ we have $\thc(E,\phi)=f^A\thc(E',\phi')$;
\item
for a morphism of smooth varieties $r\colon X'\to X$ we have $r^A\thc(E,\phi)=\thc(r^*(E,\phi))$;
\item
homomorphisms
$-\cup \thc(E,\phi)\colon A^{[\star]}_{\hphantom{[}*}(X)\to A^{[\star+n]}_{\hphantom{[}*}(\Th(E))$
are isomorphisms;
\item
we have
\[
\thc(E_1\oplus E_2, \phi_1\perp \phi_2)=q_1^A\thc(E_1,\phi_1)\cup q_2^A\thc(E_2,\phi_2),
\]
where $q_1,q_2$ are projections from $E_1\oplus E_2$ onto its summands;
\item
for the zero bundle $\mathbf{0}\to \pt$ over the point we have $\thc(\mathbf{0})=1_\pt\in A^{[0]}_{\hphantom{[}0}(\pt)$;
\item 
for the hyperbolic bundle $(H,\phi)=\left(\triv_\pt\oplus \triv_\pt, \begin{pmatrix} 0 & 1 \\ -1 & 0 \end{pmatrix}\right)$ of rank two over the point we have 
\[
\thc(H,\phi)=\Sigma^2_T 1_\pt\in A^{[2]}_{\hphantom{[}0}(T\wedge T). 
\]
\end{enumerate}
The class $\thc(E,\phi)$ is the \textit{Thom class} of the symplectic bundle.
\end{definition}

%\begin{remark}
%One can introduce analogous notions of orientations on a ring cohomology theory $A(-)$ defined on the category of pairs $(X,U)$ with a smooth variety $X$ and open subvariety $U$ repeating the above definitions and writing $A(E,E-X)$ for $A(Th(E))$, details could be found in \cite{PS}. Note that one does not have canonical suspension elements in $A(\A^n,\A^n-\{0\})$ analogous to $\Sigma^n_T 1$, so the normalization property involving $\Sigma_T 1$ from Definition~\ref{def_SLorient} and the corresponding property for the normalized symplectic orientation involving $\Sigma^2_T 1$ should be dropped. 
%\end{remark}

\begin{lemma}
\label{lem_symbol_via_Thom}
Let $A\in \SH(k)$ be a commutative ring $T$-spectrum representing an $\SL$-oriented cohomology theory. Let $X$ be a smooth variety and $f\in k[X]^*$ be an invertible function. Put $f_T\colon\triv_X \to \triv_X$ for the isomorphism given by $f_T(x,t)=(x,f(x)t)$. Then
\[
\langle f\rangle_A\cup \thc(\triv_X,\id)=\thc(\triv_X,f_T).
\]
\end{lemma}
\begin{proof}
Functoriality of Thom classes together with the normalization property yields $\thc(\triv_X,\id)=\Sigma_T 1_X$ and $\thc(\triv_X,f_T)=f_T^A(\Sigma_T 1_X)$. Then
\[
\langle f\rangle_A\cup \thc(\triv_X,\id)=\Sigma^{-1}_T(f^A_T(\Sigma_T 1_X))\cup \Sigma_T 1_X = f_T^A(\Sigma_T 1_X)=\thc(\triv_X,f_T). \qedhere
\]
\end{proof}

%\begin{remark}
%The above definition is consistent for every cohomology theory possessing Thom classes for the trivialized linear bundles, i.e. $A(X)$-module isomorphisms $A(X)\cong A(X\times \A^1,X\times \Gm)$.
%\end{remark}

Every symplectic bundle $(E,\phi)$ has a trivialization of the determinant given by Pfaffian and can be treated as a special linear bundle $(E,\lambda)$. More precisely, for a trivialized symplectic bundle $(\triv_X^{\oplus 2n}, \phi)$, the Pfaffian $\operatorname{Pf}(\phi)$ can be considered as an isomorphism $\det \triv_X^{\oplus n}=\triv_X\xrightarrow{\simeq} \triv_X$. Using this isomorphism locally and applying the well-known property $\operatorname{Pf}(B^TAB)=\det(B)\operatorname{Pf}(A)$ which, in particular, yields that these isomorphisms do not depend on the choice of trivialization, we obtain the desired isomorphism $\lambda\colon\det E\xrightarrow{\simeq} \triv_X$. Thus a special linear orientation on a ring cohomology theory induces a symplectic one, and a general orientation induces a special linear one. Hence one has a variety of examples of $\SL$-oriented cohomology theories arising from the oriented ones. Typical examples of $\SL$-oriented but not oriented cohomology theories are hermitian $K$-theory introduced by Schlichting \cite{Sch10a,Sch12} and derived Witt groups defined by Balmer \cite{Bal99,Bal05}, we recall the related constructions in the next section. Further examples are given by the algebraic cobordism $\MSL$ and $\MSp$ \cite{PW3}. We briefly recall some definitions related to the spectrum $\MSL$.

\begin{definition}
Let $E(n,m)$ be the tautological vector bundle of rank $n$ over the Grassmannian $\Gr(n,m)$. The special linear Grassmannian is the complement to the zero section of $\det E(n,m)$,
\[
\SGr(n,m)=(\det E(n,m))^0.
\]
For the canonical projection $p\colon \SGr(n,m)\to \Gr(n,m)$ equip 
\[
\Tc(n,m)=p^*E(n,m)
\] 
with the obvious trivialization $\lambda\colon \det \Tc(n,m)\xrightarrow{\simeq} \triv_{\SGr(n,m)}$ and refer to it as the \textit{tautological special linear bundle of rank $n$ over $\SGr(n,m)$}. Let $\Th(n,m)=\Th(\Tc(n,m))$ be its Thom space. For the tower of the natural monomorphisms 
\[
\hdots\to \Th(n,m)\to \Th(n,m+1) \to \Th(n,m+2)\to\hdots
\]
denote
\[
\MSL_n=\varinjlim_{m\in \mathbb{N}}\Th(n,m).
\]
In order to write down the bonding maps and the monoid structure it is more convenient to work in the category of symmetric $T^2$-spectra, see \cite{PW3} for the details. Sometimes one prefers to work with the spectrum $\MSL^{fin}$ with $\MSL_n^{fin}=\Th(n,n^2)$. The natural map $\MSL^{fin}\to \MSL$ becomes an isomorphism in $\SH(k)$.

\end{definition}

The cobordism cohomology theories are the universal ones in the sense of the following theorems (see \cite[Theorems 12.2, 13.2, 5.9]{PW3}).

\begin{theorem}
\label{thm_univMSp}
Suppose $A$ is a commutative ring $T$-spectrum equipped with a normalized symplectic orientation on $A^{[\star]}_{\hphantom{[}*}(-)$ given by Thom classes $\thc^A(E,\theta)$. Then there exists a unique morphism $\varphi^{\Sp}\colon \MSp\to A$ in $\SH(k)$ such that for every symplectic bundle $(E,\theta)$ over every smooth variety $X$  one has 
\[
\varphi_{\Th(E)}^{\Sp}(\thc^{\MSp}(E,\theta))=\thc^A(E,\theta).
\]
Moreover, this morphism is a morphism of commutative ring $T$-spectra.
\end{theorem}

\begin{theorem}
\label{thm_univMSL}
Suppose $A$ is a commutative ring $T$-spectrum equipped with a normalized special linear orientation on $A^{[\star]}_{\hphantom{[}*}(-)$ given by Thom classes $\thc^A(E,\lambda)$. Then there exists a morphism $\varphi^{\SL}\colon \MSL\to A$ in $\SH(k)$ such that 
\[
\varphi_{\Th(E)}^{\SL}(\thc^{\MSL}(E,\lambda))=\thc^A(E,\lambda)
\]
for every special linear bundle $(E,\lambda)$. It satisfies $\varphi_{\pt}^{\SL}(1^{\MSL}_{\pt})=1^A_{\pt}$.
\end{theorem}

There is a $\lim^1$-obstruction for the morphism from Theorem~\ref{thm_univMSL} to be a morphism of commutative ring $T$-spectra. As we will see shortly, this obstruction vanishes for small pointed motivic spaces.

\begin{definition}
A pointed motivic space $Y$ is called \textit{small} if $\Hom_{\SH(k)}(\Sigma^\infty_T Y,-)$ commutes with arbitrary coproducts.
\end{definition}

\begin{lemma}
\label{lem_smallMult}
For a morphism $\varphi^{\SL}$ from Theorem~\ref{thm_univMSL}, small pointed motivic spaces $Y,Y'$ and arbitrary elements $\alpha\in \MSL^{[\star]}_{\hphantom{[}*}(Y),\alpha'\in \MSL^{[\star]}_{\hphantom{[}*}(Y')$ one has 
\[
\varphi^{\SL}_{Y\wedge Y'}(\alpha\times \alpha')=\varphi^{\SL}_{Y}(\alpha)\times \varphi^{\SL}_{Y'}(\alpha').
\]
In particular, $\varphi^{\SL}_Y$ is a homomorphism of $\pi^{[\star]}_{\hphantom{[}*}(\pt)$-algebras.
\end{lemma}
\begin{proof}
Suspending one may pass to 
\[
\Sigma^{p,q}\alpha\in \MSL^{[n]}_{\hphantom{[}0}(Y\wedge S^{p,q}), \quad\Sigma^{p',q'}\alpha'\in \MSL^{[n']}_{\hphantom{[}0}(Y'\wedge S^{p',q'}).
\]
Since $\MSL\cong \MSL^{fin}$ in $\SH(k)$ we have
\begin{multline*}
\MSL^{[n]}_{\hphantom{[}0}(Y\wedge S^{p,q})= \Hom_{\SH(k)}(\Sigma_T^\infty (Y\wedge S^{p,q}),\Sigma^{2n,n} \MSL)\cong \\
\cong \Hom_{\SH(k)}(\Sigma_T^\infty (Y\wedge S^{p,q}),\Sigma^{2n,n}\MSL^{fin})
\cong \Hom_{\SH(k)}(\Sigma^{-2n,-n}\Sigma_T^\infty (Y\wedge S^{p,q}),\MSL^{fin}).
\end{multline*}
By \cite[Theorem 5.2]{V} there is a canonical isomorphism
\begin{multline*}
\Hom_{\SH(k)}(\Sigma^{-2n,-n}\Sigma_T^\infty (Y\wedge S^{p,q}),\MSL^{fin})
\cong\varinjlim_{l\in \mathbb{N}} \Hom_{\Hp(k)}(Y\wedge S^{p,q}\wedge T^{\wedge l-n}, \Th(l,l^2))=\\
=\varinjlim_{i\in \mathbb{N}} \Hom_{\Hp(k)}(Y\wedge S^{p,q}\wedge T^{\wedge i}, \Th(n+i,(n+i)^2)).
\end{multline*}
Note that going through the above isomorphisms we obtain a map 
\[
\Hom_{\Hp(k)}(Y\wedge S^{p,q}\wedge T^{\wedge i}, \Th(n+i,(n+i)^2)) \to \MSL^{[n]}_{\hphantom{[}0}(Y\wedge S^{p,q})
\]
given by 
\[
f\mapsto \Sigma^{-i}_T(\thc^{MSL}(\Tc(n+i,(n+i)^2))\circ \Sigma_T^{\infty} f)
\]
with the Thom class $\thc^{MSL}(\Tc(n+i,(n+i)^2))$ arising as the tautological morphism 
\[
\Sigma^\infty_T \Th(n+i,(n+i)^2)\to \Sigma^{2(n+i),n+i} \MSL.
\]
Thus there exists some $i\in \mathbb{N}$ and 
\[
f \in \Hom_{\Hp(k)} (Y\wedge S^{p,q}\wedge T^{\wedge i},  \Th(n+i,(n+i)^2))
\] 
such that
\[
\Sigma_T^{i}\Sigma^{p,q}\alpha=f^{\MSL}\thc^{\MSL}(\Tc(n+i,(n+i)^2)).
\]
By a similar argument we obtain that
\[
\Sigma_T^{i'}\Sigma^{p',q'}\alpha'=f'^{\MSL}\thc^{\MSL}(\Tc(n'+i',(n'+i')^2))
\]
for some $i',p',q'\in \mathbb{N}$ and $f\in \Hom_{\Hp (k)}(Y'\wedge S^{p',q'}\wedge T^{\wedge i'},\Th(n'+i',(n'+i')^2))$. Set $m=n+i, m'=n'+i'$ and consider the following diagram.

\[
\xymatrix{
\Sigma_T^{-m}\Sigma_T^{\infty}\MSL_m^{fin}\wedge \Sigma_T^{-m'}\Sigma_T^{\infty}\MSL_{m'}^{fin} \ar[rr]^(0.65){\thc^{A}_m\wedge \thc^{A}_{m'}} \ar[dd]_{\mu_{m,m'}} \ar[dr]^(0.6){\quad \thc^{\MSL}_m\wedge \thc^{\MSL}_{m'}} & & A\wedge A \ar[dd]^{\mu_{A}} \\
&  \MSL\wedge \MSL \ar[ur]_{\varphi^{\SL}\wedge \varphi^{\SL}} \ar[dd]^(0.3){\mu_{\MSL}} & \\
\Sigma_T^{-m-m'}\Sigma_T^{\infty}\MSL_{m+m'}^{fin} \ar[rr]^{\thc^{A}_{m+m'}} \ar[dr]_{\thc^{\MSL}_{m+m'}} &  & A  \\
 & \MSL\ar[ur]_{\varphi^{\SL}} &
}
\]
Here $\mu_{m,m'}$ is induced by the canonical embedding
\[
\Tc(m,m^2)\times \Tc(m',m'^2)\to \Tc(m+m',(m+m')^2)
\]
and all the maps $\thc$ are given by the corresponding Thom classes, for example $\thc_{m}^{A}$ is induced by 
\[
\thc^{A}(\Tc(m,m^2))\in A^{[m]}_{\hphantom{[}0}(\Th(m,m^2))\cong \Hom_{\SH(k)}(\Sigma_T^{-m}\Sigma^\infty_T \MSL^{fin}_m,A).
\]
The left and the back squares commute by the multiplicativity property of the Thom classes. The triangles commute by the choice of $\varphi^{\SL}$. Hence
\begin{multline*}
\mu_{A}(\varphi^{\SL}\wedge\varphi^{\SL})(\thc^{\MSL}_m\wedge \thc^{\MSL}_{m'})(\Sigma_T^{-m}(\Sigma_T^\infty f)\wedge \Sigma_T^{-m'}(\Sigma_T^\infty f'))=\\
=\varphi^{\SL} \mu_{\MSL}(\thc^{\MSL}_m\wedge \thc^{\MSL}_{m'})(\Sigma_T^{-m}(\Sigma_T^\infty f)\wedge \Sigma_T^{-m'}(\Sigma_T^\infty f')),
\end{multline*}
%Computing both sides at $th^{MSL}(\Tc(m,m^2))\times th^{MSL}(\Tc(m',m'^2))$
and, omitting indices at $\varphi^{\SL}$,
\[
\varphi^{\SL}(\Sigma_T^{i}\Sigma^{p,q}\alpha\times \Sigma_T^{i'}\Sigma^{p',q'}\alpha')=\varphi^{\SL}(\Sigma_T^{i}\Sigma^{p,q}\alpha)\times \varphi^{\SL}(\Sigma_T^{i'}\Sigma^{p',q'}\alpha').
\]
The morphisms $\varphi^{\SL}_Y$ and $\varphi^{\SL}_{Y'}$ respect suspension isomorphisms, so the claim follows from the above equality.

\end{proof}

\section{Schlichting's hermitian $K$-theory and Witt groups.}

In this section we recall the basic definitions related to hermitian $K$-groups (also known as higher Grothendieck-Witt groups) introduced by Schlichting \cite{Sch10a,Sch10b}. In our exposition we mainly follow \cite{Sch12} which employs the setting of dg-categories and we refer the reader to loc. cit. for the details. Recall that we are working over a field $k$ of characteristic different from $2$, thus $\tfrac{1}{2}\in \Gamma(X,\mathcal{O}_X)$ for every $X$.

For a smooth variety $X$ consider the category $\Vect(X)$ of vector bundles over $X$ and let $\mathrm{sPerf}(X)=\mathrm{Ch}^b(\Vect(X))$ be the dg-category of bounded chain complexes in $\Vect(X)$. We equip $\mathrm{sPerf}(X)$ with weak equivalences given by quasi-isomorphisms and duality consisting of the functor $ \mathcal{E}\mapsto \mathcal{E}^\vee=\Hom^\bullet_{\mathcal{O}_X}(\mathcal{E},\mathcal{O}_X)$ and the canonical double dual identification $\mathrm{can}^X_\mathcal{E}\colon \mathcal{E}\xrightarrow{\simeq} \mathcal{E}^{\vee\vee}$.

Let $Z\subset X$ be a closed subset of a smooth variety $X$ with an open complement $U=X-Z$. We denote by $\mathrm{sPerf}_Z(X)$ the full dg-subcategory of $\mathrm{sPerf}(X)$ consisting of the complexes supported on $Z$ (acyclic on $U$). This category inherits the weak equivalences and duality from $\mathrm{sPerf}(X)$.

In \cite[Definition 5.4]{Sch12} Schlichting defines the Grothendieck-Witt spectrum of a dg-category with weak equivalences and duality. For the considered categories of chain complexes we denote the corresponding spectra
\[
\KO(X)=\GW(\mathrm{sPerf}(X)), \quad \KO(X,U)=\GW(\mathrm{sPerf}_Z(X)).
\]
More generally we write 
\[
\KO^{[n]}(X)=\GW^{[n]}(\mathrm{sPerf}(X)), \quad \KO^{[n]}(X,U)=\GW^{[n]}(\mathrm{sPerf}_Z(X))
\]
for the Grothendieck-Witt spectra for the $n$-th shifted duality and 
\[
\KO^{[n]}_{\hphantom{[}i}(X)=\pi_i(\GW^{[n]}(\mathrm{sPerf}(X)),\quad \KO^{[n]}_{\hphantom{[}i}(X,U)=\pi_i(\GW^{[n]}(\mathrm{sPerf}_Z(X))
\]
for its homotopy groups. The latter groups are referred to as \textit{hermitian $K$-groups}.

There is a homotopy fibration sequence \cite[Theorem 9.5]{Sch12}
\[
\KO^{[n]}(X,U)\to \KO^{[n]}(X) \to \KO^{[n]}(U)	 
\]
and therefore long exact sequences
\[
\hdots \to \KO^{[n]}_{\hphantom{[}i}(X,U)\to \KO^{[n]}_{\hphantom{[}i}(X) \to \KO^{[n]}_{\hphantom{[}i}(U)	\xrightarrow{\partial} \KO^{[n]}_{i-1}(X,U)\to \hdots
\]

Spectra $\KO^{[n]}(X)$ and $\KO^{[n]}(X,U)$ are far from being connected. For $i<0$ groups $\KO^{[n]}_{\hphantom{[}i}(X)$ and $\KO^{[n]}_{\hphantom{[}i}(X,U)$ can be identified with the derived Witt groups of Balmer \cite{Bal99} which are defined as follows. Let $\mathrm{D}^b(\Vect(X))$ and $\mathrm{D}^b_Z(\Vect(X))$ be the derived categories of $\mathrm{Ch}^b(\Vect(X))$ and $\mathrm{Ch}^b_Z(\Vect(X))$ respectively. Equip these categories with the dualities consisting of the functor $ (-)^\vee$ and the canonical double dual identification $\mathrm{can}^X_\mathcal{E}$ as above. We denote by $\mathrm{D}^b(\Vect(X))[n]$ and  $\mathrm{D}^b_Z(\Vect(X))[n]$ the same derived categories but with $n$-th shifted dualities.

A symmetric object for a triangulated category $C$ with a duality is a pair $(A,\alpha)$ of $A\in \operatorname{Ob}(C)$ and a morphism $\alpha\colon A\to A^\vee$ agreeing with the duality. A symmetric space $(A,\alpha)$ is a symmetric object with $\alpha$ being an isomorphism. There are obvious notions of the isomorphism of symmetric spaces and of the orthogonal sum $(A,\alpha)\perp (B,\beta)$ of symmetric spaces. For every symmetric object $(A,\alpha)$ there exists a canonical symmetric space $\Cone(A,\alpha)$ for the $1$-st shifted duality, in particular, for every object $A\in \operatorname{Ob}(C)$ there is a hyperbolic symmetric space $\operatorname{H}(A)=\Cone(A,0)$. The Grothendieck-Witt group $\GW(C)$ of a small triangulated category with duality is the quotient of the free abelian group on the isomorphism classes of symmetric spaces by the relations $[(A,\alpha)\perp (B,\beta)]=[(A,\alpha)]+[(B,\beta)]$ and $\Cone(A,\alpha)=\operatorname{H}(A)$, where for the second class of relations we regard $A$ as an object of the triangulated category with $(-1)$-st shifted duality. The Witt group $\W(C)$ is the quotient of the Grothendieck-Witt group by the subgroup generated by the classes of the hyperbolic spaces.

We write 
\[
\GW^{n}(X)=\GW(\mathrm{D}^b(\Vect(X))[n]),\quad \GW^{n}(X,U)=\GW(\mathrm{D}^b_Z(\Vect(X))[n]),
\]
\[
\W^{n}(X)=\W(\mathrm{D}^b(\Vect(X))[n]),\quad \W^{n}(X,U)=\W(\mathrm{D}^b_Z(\Vect(X))[n])
\]
for the corresponding Grothendieck-Witt and Witt groups arising from the derived category of the vector bundles over a smooth variety $X$. For $i<0$ there are natural identifications \cite[Proposition 5.6, 6.3]{Sch12} 
\begin{gather*}
\Theta\colon \KO^{[n]}_{\hphantom{[}0}(X)\xrightarrow{\simeq} \GW^n(X), \quad \Theta\colon \KO^{[n]}_{\hphantom{[}0}(X,U)\xrightarrow{\simeq} \GW^n(X,U),\\
\Theta\colon \KO^{[n]}_{\hphantom{[}i}(X)\xrightarrow{\simeq} \W^{n-i}(X),\quad \Theta\colon \KO^{[n]}_{\hphantom{[}i}(X,U)\xrightarrow{\simeq} \W^{n-i}(X,U).
\end{gather*}

For a smooth variety $X$ and an invertible function $s\in k[X]^*$ we write $\langle s \rangle$ for the class 
\[
\langle s \rangle = [(\triv_X,s)]\in \GW^0(X)
\] 
and 
\[
\langle s \rangle_{\KO}= \Theta^{-1}(\langle s \rangle)\in \KO^{[0]}_0(X)
\]
for the corresponding element in hermitian $K$-theory. Let $X$ be a smooth variety and $U,V\subset X$ be open subsets. Tensor product of complexes gives rise to a pairing \cite[Section 5]{Sch12}
\[
\multko \colon \KO^{[n]}_{\hphantom{[}i}(X,U)\times \KO^{[m]}_{\hphantom{[}j} (X,V)\to \KO^{[n+m]}_{\hphantom{[}i+j}(X, U\cup V),
\]
which is graded-commutative in the following sense:
\[
a\multko b = (-1)^{ij}\langle -1 \rangle_{\KO}^{mn} \multko b\multko a
\]
for $a\in \KO^{[n]}_{\hphantom{[}i}(X,U)$ and $b\in \KO^{[m]}_{\hphantom{[}j} (X,V)$. This pairing coincides with the partially defined one of Panin and Walter \cite[4.6a, 4.6b]{PW2} which in turn is a slight generalization of the pairing constructed by Gille and Nenashev \cite{GN} for the derived Witt groups. We briefly recall that one can define Gille-Nenashev paring
\[
\boxtimes\colon \GW^n (X,U)\times \GW^m (X,V)\to \GW^{n+m}  (X, U\cup V)
\]
via
\[
[(A_\bullet,\alpha)]\boxtimes [(B_\bullet,\beta)]= [(A_\bullet\otimes B_\bullet, \alpha\widetilde{\otimes} \beta)]\in \GW^{n+m}(X, U\cup V),
\]
where $\alpha\widetilde{\otimes} \beta$ equals to $\alpha\otimes\beta$ up to some signs and identification 
\[
A_\bullet^\vee[n]\otimes B_\bullet^\vee[m]=(A_\bullet\otimes B_\bullet)^\vee[n+m].
\]
This pairing respects hyperbolic spaces and induces a pairing
\[
\boxtimes\colon \W^n (X,U)\times \W^m (X,V)\to \W^{n+m}  (X, U\cup V).
\]
Recall that for $a\in \W^n (X,U), b\in \W^m (X,V)$ one has $a\boxtimes b= (-1)^{nm} b\boxtimes a$.

The identification $\Theta\colon \KO^{[n]}_{\hphantom{[}i}\xrightarrow{\simeq} \W^{n-i}$, $i\le 0$, respects multiplication only up to a certain sign specified in the following lemma.

\begin{lemma}\label{lem_KOGW}
Let $X$ be a smooth variety and $U,V\subset X$ be open subsets. Then for $n,m\in \Z$, $i,j\le 0$ the following diagram commutes
\[
\xymatrix{
\KO^{[n]}_{\hphantom{[}i}(X,U) \times \KO^{[m]}_{\hphantom{[}j}(X,V) \ar[r]^(0.55){\multko} \ar[d]^{\Theta\times \Theta}_{\cong} & \KO^{[n+m]}_{\hphantom{[}i+j}(X,U\cup V) \ar[d]^{(-1 )^{(m-j)i} \Theta}_{\cong} \\
\W^{n-i}(X,U) \times \W^{m-j}(X,V) \ar[r]^(0.55){\boxtimes} & \W^{n+m-i-j}(X,U\cup V),
}
\]
i.e. for $a\in \KO^{[n]}_{\hphantom{[}i}(X,U)$, $b\in \KO^{[m]}_{\hphantom{[}j}(X,V)$ one has
\[
\Theta(a\multko b)=( -1 )^{(m-j)i}  \Theta(a) \boxtimes \Theta(b).
\]
\end{lemma}
\begin{proof}
The case of $i=j=0$ follows from \cite[Section 5]{Sch12}. Recall that by \cite[Proposition 6.3]{Sch12} for $i=-1$ isomorphisms $\Theta$ are given by an exact sequence
\[
\mathrm{K}_0 \xrightarrow{H} \KO^{[n]}_{\hphantom{[}0} \xrightarrow{\multko \eta_S} \KO^{[n-1]}_{\hphantom{[}-1} \to 0
\]
combined with the identification $\Theta\colon  \KO^{[n]}_{\hphantom{[}0} \xrightarrow{\simeq} \GW^{n}$. Here $H$ is the hyperbolic map and $\eta_S$ is a certain element of $\KO^{[-1]}_{\hphantom{[}-1}(\pt)$. For $i<-1$ one uses the case of $i=-1$ and isomorphisms $\multko \eta_S\colon \KO^{[n+1]}_{\hphantom{[}i+1}\xrightarrow{\simeq} \KO^{[n]}_{\hphantom{[}i}$. Hence there exist $\widetilde{a}\in \KO^{[n-i]}_{\hphantom{[}0}(X,U)$ and $\widetilde{b}\in \KO^{[m-j]}_{\hphantom{[}0}(X,V)$ such that $a=\widetilde{a}\multko \eta_S^{-i}, b=\widetilde{b}\multko \eta_S^{-j}$. Moreover, we have $\Theta(a)=\Theta(\widetilde{a})$ and $\Theta(b)=\Theta(\widetilde{b})$. Thus
\begin{multline*}
\Theta(a\multko b)= \Theta(\widetilde{a}\multko \eta_S^{-i}\multko \widetilde{b}\multko \eta_S^{-j})= \Theta(\langle -1 \rangle^{(m-j)i}_{\KO} \widetilde{a}\multko \widetilde{b}\multko \eta_S^{-i-j})
=\Theta( \langle -1 \rangle^{(m-j)i}_{\KO}\multko \widetilde{a}\multko \widetilde{b})
=\\
=\Theta(\langle -1 \rangle^{(m-j)i}_{\KO})\boxtimes \Theta( \widetilde{a})\boxtimes \Theta(\widetilde{b})
= \Theta(\langle -1 \rangle^{(m-j)i}_{\KO})\boxtimes \Theta( a)\boxtimes \Theta(b) = ( -1 )^{(m-j)i} \Theta( a)\boxtimes \Theta(b).
\end{multline*}
The last equality follows from the fact that $\langle -1\rangle = -1$ in the Witt group.
\end{proof}

For a special linear bundle $(E,\lambda)$ of rank $n$ over a smooth variety $X$ one can construct a Thom class for hermitian $K$-theory using the method introduced by Nenashev for Witt groups \cite{Ne07}. Let $p\colon E\to X$ be the structure map. Consider the pullback $p^*E=E\oplus E\to E$. There is a canonical diagonal section $s\colon E\to E\oplus E$ that defines a map $p^*E^\vee\to \triv_E$ via the pairing $p^*E\otimes p^*E^\vee\to \triv_E$. This map give rise to the Koszul complex 
\[
K(E)=(0\to \Lambda^np^*E^\vee\to \Lambda^{n-1}p^*E^\vee\to\ldots\to\Lambda^2p^*E^\vee\to p^*E^\vee\to\triv_E\to 0)
\]
which is treated as a chain complex located in homological degrees $n$ through $0$. It is well known that this complex is supported on $X$. The canonical isomorphism 
\[
\Xi(E)\colon K(E)\xrightarrow{\simeq}K(E)^\vee[n] \otimes \det p^*E^\vee
\]
combined with $\lambda$ induces an isomorphism 
\[
\Xi(E,\lambda)\colon K(E)\xrightarrow{\simeq} K(E)^\vee[n].
\]
One can show that $(K(E),\Xi(E,\lambda))$ is a symmetric space for the the $n$-th shifted duality. Denote
\[
\thc^{\GW}(E,\lambda)=[K(E),\Xi(E,\lambda)]\in \GW^n(E,E-X)
\] 
the corresponding element in the Grothendieck-Witt group. The element 
\[
\thc^{\KO}(E,\lambda)=\Theta^{-1}(\thc^{\GW}(E,\lambda))\in \KO_{\hphantom{[}0}^{[n]}(E,E-X)
\]
represents the Thom class of the special linear bundle $(E,\lambda)$ for the hermitian $K$-theory, in particular, the homomorphisms 
\[
-\multko \thc^{\KO}(E,\lambda)\colon \KO^{[m]}_{\hphantom{[}i}(X)\xrightarrow{\simeq} \KO_{\hphantom{[}i}^{[m+n]}(E,E-X).
\]
are isomorphisms \cite[Theorem 5.1]{PW2}.

We finish this section with the following straightforward computations.

\begin{lemma}
\label{lem_KOsymbol}
Let $X$ be a smooth variety and $f\in k[X]^*$. Denote $f_T\colon \triv_X\to\triv_X$ the isomorphism given by $f_T(x,t)=(x,f(x)t)$. Then 
\[
\langle f \rangle_{\KO} \multko \thc^{\KO}(\triv_X,\id)=\thc^{\KO}(\triv_X,f_T).
\]
\end{lemma}
\begin{proof}
Follows from the construction of the Thom classes, Lemma~\ref{lem_KOGW} and straightforward computation of tensor product. 
\end{proof}

\begin{lemma}
\label{lem_connwitt}
Let $\partial \colon \GW^0(\Gm) \to \W^1(\A^1,\Gm)$ be the connecting homomorphism in the localization sequence for the embedding $\{0\}\to \A^1$ and let $t\in k[\Gm]^*$ be the coordinate function. Then one has 
\[
\partial(\langle t \rangle) = \thc^\W(\triv_{\pt},\id),
\]
where $\thc^\W(\triv_{\pt},\id)$ is the Thom class $\thc^\GW(\triv_{\pt},\id)\in \GW^1(\A^1,\Gm)$ considered as an element of $\W^1(\A^1,\Gm)$.
\end{lemma}
\begin{proof}
In order to compute $\partial(\langle t \rangle)$ one should write down the cone for the symmetric space $(A_\bullet,t)$ with $A_\bullet$ being the complex concentrated in the zeroth degree $A_0=\triv_{\A^1}$. A straightforward computation shows that this cone coincides with the desired Thom class, $[(\triv_{\A^1}\xrightarrow{t}\triv_{\A^1}, \Xi(\triv_{\A^1},\id))]$.
\end{proof}

\section{$T$-spectrum $\BO$ and the cohomology theory $\BO^{[\star]}_{\,\,*}$.}

In \cite{PW2} Panin and Walter constructed a commutative ring $T$-spectrum $\BO$ representing hermitian $K$-theory in the following precise sense (see \cite[Theorem 1.3-1.6]{PW2} and \cite[Lemma 4.4]{PW4}):

\begin{theorem} \label{thm_BOKO}
For every smooth variety $X$ and open subset $U$ and for every $i,n\in \mathbb{Z}$ there exist canonical functorial isomorphisms $\gamma\colon \BO^{[n]}_{\hphantom{[}i}(X/U)\xrightarrow{\simeq} \KO_{\hphantom{[}i}^{[n]}(X,U)$. 
\begin{enumerate}
\item
These isomorphisms agree with connecting homomorphisms $\partial$ in localization sequences. 
\item
The $\cup$-product on $\BO^{[\star]}_{\hphantom{[}*}(-)$ arising from the monoid structure of $\BO$ agrees with the $\multko$-product on $\KO_{\hphantom{[}*}^{[\star]}(-)$. 
\item
$\BO^{[\star]}_{\hphantom{[}*}(-)$ is equipped with Thom classes $\thc^{\BO}(E,\lambda)$ defined for special linear bundles. These classes satisfy $\gamma(\thc^{\BO}(E,\lambda))=\thc^{\KO}(E,\lambda)$ and $\thc^{\BO}(\triv_X,\id)=\Sigma_T 1_X$. 
\item
$\gamma(1)=\langle 1\rangle_{\KO}$ and $\gamma(\epsilon)=\langle -1\rangle_{\KO}$.
\end{enumerate}
\end{theorem}

\begin{corollary}
\label{cor_BOsymb}
For a smooth variety $X$ and $f\in k[X]^*$ we have 
\[
\gamma(\langle f\rangle_{\BO})=\langle f \rangle_{\KO}.
\]
\end{corollary}
\begin{proof}
By the above theorem combined with Lemmas~\ref{lem_symbol_via_Thom} and~\ref{lem_KOsymbol} we have
\begin{multline*}
\gamma(\langle f\rangle_{\BO})\multko \thc^{\KO}(\triv_X,\id)=\gamma(\langle f\rangle_{\BO}\cup \thc^{\BO}(\triv_X,\id))=\\
=\gamma(\thc^{\BO}(\triv_X,f_T))=\thc^{\KO}(\triv_X,f_T)=\langle f \rangle_{\KO}\multko \thc^{\KO}(\triv_X,\id).
\end{multline*}
Canceling $\thc^{\KO}(\triv_X,\id)$ we obtain the claim.
\end{proof}

\begin{lemma}
\label{lem_etaW}
For the stable Hopf map $\eta\in \BO^{[-1]}_{\hphantom{[}-1}(\pt)$ we have $\gamma(\eta)=\eta_S$, where $\eta_S\in \KO^{[-1]}_{\hphantom{[}-1}(\pt)$ is the element corresponding to $1\in \W^{0}(\pt)\cong \KO^{[-1]}_{\hphantom{[}-1}(\pt)$ (see the proof of Lemma~\ref{lem_KOGW} and \cite[Sections~6,7]{Sch12}).
\end{lemma}
\begin{proof}
Consider the following diagram.
\[
\xymatrix@C=5pc{
\BO^{[0]}_{\hphantom{[}0}(\Gm) \ar[r]^\partial \ar[d]_{\gamma}^{\simeq} & \BO^{[0]}_{-1}(T) \ar[d]_{\gamma}^{\simeq} & \BO^{[-1]}_{\hphantom{[}-1}(\pt) \ar[l]_{ \Sigma_T}  \ar[d]_{\gamma}^{\simeq} \\
\KO^{[0]}_{\hphantom{[}0}(\Gm) \ar[r]^\partial \ar[d]_{\Theta}^{\simeq} & \KO^{[0]}_{-1}(\A^1,\Gm) \ar[d]_{\Theta}^{\simeq}  & \KO^{[-1]}_{\hphantom{[}-1}(\pt) \ar[l]_{\cup \thc^{\KO}(\triv_\pt,\id)} \ar[d]_{-\Theta}^{\simeq} \\
\GW^{0}(\Gm) \ar[r]^\partial & \W^1(\A^1,\Gm) & \W^{0}(\pt) \ar[l]_{\boxtimes \thc^{\GW}(\triv_\pt,\id)}
}
\]
Here $\partial$ denotes the connecting homomorphisms in the localization sequences for the embedding $\{0\}\to \A^1$.
The top half commutes by Theorem~\ref{thm_BOKO}, the lower right square commutes by Lemma~\ref{lem_KOGW}, the lower left square commutes since $\Theta$ is induced by a natural morphism of spectra \cite[Section~7]{Sch12}. 
By Theorem~\ref{thm_etacoh} we have
\[
(\Theta\gamma(\eta))\boxtimes \thc^{\GW}(\triv_\pt,\id)=(\Theta\gamma\Sigma_T^{-1}\partial (\langle -t^{-1}\rangle_{\BO}))\boxtimes \thc^{\GW}(\triv_\pt,\id).
\]
Using the commutativity of the above diagram we obtain
\[
(\Theta\gamma\Sigma_T^{-1}\partial (\langle -t^{-1}\rangle_{\BO}))\boxtimes \thc^{\GW}(\triv_\pt,\id)=-\partial\Theta\gamma(\langle -t^{-1}\rangle_{\BO})
\]
Lemma~\ref{lem_connwitt} together with Corollary~\ref{cor_BOsymb} yield
\[
-\partial\Theta\gamma(\langle -t^{-1}\rangle_{\BO})=-\partial(\langle -t^{-1}\rangle)=-\partial(\langle -t\rangle)=-\langle-1\rangle\boxtimes\partial(\langle t\rangle)=\thc^\W(\triv_\pt,\id).
\]
Combining the above equalities we obtain
\[
(\Theta\gamma(\eta))\boxtimes \thc^{\GW}(\triv_\pt,\id)=\thc^\W(\triv_\pt,\id)=\Theta(\eta_S) \boxtimes \thc^{\GW}(\triv_\pt,\id).
\]
The claim follows via cancellation of the Thom classes and $\Theta$.
\end{proof}

The main result of this section states that inverting the stable Hopf map in $\BO^{[\star]}_{\hphantom{[}*}(X)$ one obtains the Laurent polynomial ring over $\W^\star(X)$ in $\eta$ with a certain commutativity condition. The more precise statement follows. 

\begin{definition}
For a smooth variety $X$ let $\W^{\star}(X)[\eta,\eta^{-1}]$ be a bigraded $(-1,-1)$--commutative Laurent polynomial algebra with $\deg \eta=(-1,-1)$ and $\deg a=(n,0)$ for $a\in \W^{n}(X)$, i.e. for $a\in \W^n(X),b\in \W^m(X)$ one has
\[
(a\cdot\eta^i)\cdot(b\cdot\eta^j)=(-1)^{(n-i)(m-j)+ij} (b\cdot\eta^j)\cdot (a\cdot\eta^i).
\]
We use Gille-Nenashev pairing on $\W^{\star}(X)$, i.e. $a\cdot b=a\boxtimes b$ for $a,b$ as above.

Define $\psi\colon \BO^{[\star]}_{\hphantom{[}*}(X)\to \W^{\star}(X)[\eta,\eta^{-1}]$ in the following way. For $a\in \BO^{[n]}_{\hphantom{[}i}(X)$ put
\[
\psi(a)=\left\{
\begin{array}{ll}
\Theta\gamma(a)\cdot \eta^{-i}, & i<0\\
\Theta\gamma(a\cup\eta^{2i+2})\cdot\eta^{-i}, & i\ge 0.
\end{array}\right.
\]
We claim that $\psi$ is a homomorphism of bigraded algebras. One can easily see that $\psi$ is additive and respects the grading. In order to check the multiplicativity property take $a\in \BO^{[n]}_{\hphantom{[}i}(X), b\in \BO^{[m]}_{\hphantom{[}j}(X)$. The case of $i<0, j<0$ follows from Lemma~\ref{lem_KOGW} and Theorem~\ref{thm_BOKO}:
\begin{multline*}
\psi(a\cup b)=\Theta\gamma(a\cup b)\cdot\eta^{-i-j}=((-1)^{(m-j)i}\Theta\gamma(a)\boxtimes \Theta\gamma(b))\cdot\eta^{-i-j}=\\
=(-1)^{(m-j)i}(-1)^{-(m-j)i}\Theta\gamma(a)\cdot\eta^{-i}\cdot\Theta\gamma(b)\cdot\eta^{-j}=\psi(a)\cdot\psi(b).
\end{multline*}
The remaining cases are quite the same and straightforward, so we leave the detailed check to the reader.
\end{definition}

\begin{theorem}
\label{thm_KOW}
The morphism $\psi$ induces an isomorphism of bigraded algebras
\[
\widetilde{\psi}\colon \BO^{[\star]}_{\hphantom{[}*}(X)[\eta^{-1}]\xrightarrow{\simeq} \W^{\star}(X)[\eta,\eta^{-1}].
\]
\end{theorem}
\begin{proof}
By Lemma~\ref{lem_etaW} we know that 
\[
\psi(\eta)=\Theta\gamma(\eta)\cdot \eta=\eta
\]
thus $\psi(\eta)$ is invertible and $\psi$ induces a homomorphism
\[
\widetilde{\psi}\colon \BO^{[\star]}_{\hphantom{[}*}(X)[\eta^{-1}]\xrightarrow{} \W^{\star}(X)[\eta,\eta^{-1}].
\]
Recall that $\psi\colon \BO^{[n]}_{\hphantom{[}i}(X)\to (\W^{\star}(X)[\eta,\eta^{-1}])^{n,i}$ is an isomorphism for $i<0$, thus $\widetilde{\psi}$ is an isomorphism as well. \qedhere
\end{proof}

\section{A motivic variant of a theorem by Conner and Floyd.}

For this section fix the canonical morphisms 
\[
\varphi^{\Sp}\colon \MSp \to \BO,\quad\psi^{\Sp}\colon \MSp \to \MSL
\]
given by Theorem~\ref{thm_univMSp} and a morphism 
\[
\varphi^{\SL}\colon \MSL\to \BO
\]
given by Theorem~\ref{thm_univMSL}. 

\begin{definition}
\label{def_etaloc}
Let $A$ be a commutative ring $T$-spectrum and $Y$ be a pointed motivic space. Define
\[
A^{[\star]}_{\eta *}(Y)=A^{[\star]}_{\hphantom{[}*}(Y)\otimes_{A^{[\star]}_{\hphantom{[}*}(\pt)} A^{[\star]}_{\hphantom{[}*}(\pt)[\eta^{-1}],\quad
A^{[n]}_{\eta i}(Y)=(A^{[\star]}_{\hphantom{[} *}(Y)\otimes_{A^{[\star]}_{\hphantom{[} *}(\pt)}A^{[\star]}_{ \hphantom{[}*}(\pt)[\eta^{-1}])^{[n]}_{\hphantom{[} i}
\]
to be the $([1],1)$-periodic ring cohomology theory obtained by inverting the stable Hopf map $\eta\in A^{[-1]}_{\hphantom{[}1}(\pt)$. In other words, we consider $A^{[\star]}_{\hphantom{[} *}(Y)$ as a $\pi^{[\star]}_{\hphantom{[}*}(\pt)$-bimodule (see Section 2) and localize at $\eta\in \pi^{[-1]}_{\hphantom{[}-1}(\pt)$.
\end{definition}

\begin{lemma}
\label{lem_etaphi}
For a small pointed motivic space $Y$ morphisms $\varphi^{\Sp}, \varphi^{\SL}, \psi^{\Sp}$ induce ring homomorphisms fitting in the commutative triangle
\[
\xymatrix{
\MSp_{\eta *}^{[\star]}(Y) \ar[r]^{\varphi^{\Sp}_{\eta,Y}} \ar[d]_{\psi^{\Sp}_{\eta,Y}}& \BO_{\eta *}^{[\star]}(Y) \\
\MSL_{\eta *}^{[\star]}(Y) \ar[ur]_{\varphi^{\SL}_{\eta,Y}}& 
}
\]
\end{lemma}
\begin{proof}
By the uniqueness part of Theorem~\ref{thm_univMSp} it follows that $\varphi^{\Sp}=\varphi^{\SL}\psi^{\Sp}$, so the following triangle commutes.
\[
\xymatrix{
\MSp^{[\star]}_{\hphantom{[}  *}(Y) \ar[r]^{\varphi^{\Sp}_{Y}} \ar[d]_{\psi^{\Sp}_{Y}}& \BO^{[\star]}_{\hphantom{[} *}(Y) \\
\MSL^{[\star]}_{ \hphantom{[} *}(Y) \ar[ur]_{\varphi^{\SL}_{Y}}& 
}
\]
The morphisms $\varphi^{\Sp},\psi^{\Sp}$ are morphisms of the monoids, thus $\varphi^{\Sp}_{Y}$ and $\psi^{\Sp}_{Y}$ are homomorphisms of $\pi^{[\star]}_{\hphantom{[} *}(\pt)$-algebras. The last morphism $\varphi^{\SL}_Y$ is a $\pi^{[\star]}_{\hphantom{[} *}(\pt)$-algebra homomorphism as well by Lemma~\ref{lem_smallMult}, so the claim follows via localization.
\end{proof}

\begin{definition}
For $A\in \SH(k)$ and a motivic space $Y$ put
\[
A^{[2\star]}_{\hphantom{[}0}(Y)=\bigoplus_{n\in\mathbb{Z}}A^{[2n]}_{\hphantom{[}0}(Y).
\]
\end{definition}

Recall the following theorem reconstructing hermitian $K$-theory via algebraic symplectic cobordism \cite[Theorem 1.1]{PW4}.
\begin{theorem}
\label{thm_sCF}
For every small pointed motivic space $Y$ morphism $\varphi^{\Sp}$ induces an isomorphism of bigraded rings
\[
\overline{\varphi}_Y^{\Sp}\colon \MSp^{[\star]}_{\hphantom{[} *}(Y)\otimes_{\MSp^{[2\star]}_{\hphantom{[} 0}(\pt)} \BO^{[2\star]}_{\hphantom{[} 0} (\pt) \xrightarrow{\simeq} \BO^{[\star]}_{\hphantom{[} *}(Y).
\]
\end{theorem}

\begin{corollary}
\label{cor_sCFW}
For every smooth variety $X$ morphism $\varphi^{\Sp}$ induces an isomorphism
\[
\MSp_{\eta *}^{[\star]}(X)\otimes_{\MSp_{\hphantom{[} 0}^{[2\star]}(\pt)} \W^{2\star} (\pt) \xrightarrow{\simeq} \W^{\star}(X)[\eta,\eta^{-1}].
\]
\end{corollary}
\begin{proof}
Isomorphism $\overline{\varphi}_Y^{\Sp}$ is an isomorphism of $\pi^{[\star]}_{\hphantom{[}*}(\pt)$-modules with the module structure on the left arising from the first factor. Inverting the stable Hopf map and applying Theorem~\ref{thm_KOW} we obtain an isomorphism
\[
\overline{\varphi}_{X}^{\Sp}\colon \MSp_{\eta *}^{[\star]}(X)\otimes_{\MSp_{\hphantom{[} 0}^{[2\star]}(\pt)} \BO_{\hphantom{[} 0}^{[2\star]} (\pt) \xrightarrow{\simeq} \W^{\star}(X)[\eta,\eta^{-1}].
\]
Consider the following commutative diagram.
\[
\xymatrix{
\MSp_{\eta *}^{[\star]}(X)\otimes_{\MSp_{\hphantom{[} 0}^{[2\star]}(\pt)} \BO_{\hphantom{[} 0}^{[2\star]} (\pt) \ar[r]^(0.65){\overline{\varphi}_{X}^{\Sp}} \ar[d]_{\pi} &  \W^{\star}(X)[\eta,\eta^{-1}] \\
\MSp_{\eta *}^{[\star]}(X)\otimes_{\MSp_{\hphantom{[} 0}^{[2\star]}(\pt)} \W^{2\star} (\pt) \ar[ur]_(0.6){\varphi} & 
}
\]
Here $\varphi$ and $\pi$ are induced by $\varphi^{\Sp}$ and the natural surjection 
\[
\BO_{\hphantom{[} 0}^{[2\star]} (\pt)\cong \GW^{2\star} (\pt)\to \W^{2\star} (\pt)
\]
respectively. The map $\overline{\varphi}_{X}^{\Sp}$ is an isomorphism and $\pi$ is surjective, thus $\varphi$ is an isomorphism.
\end{proof}

The goal of this section is to replace symplectic cobordism in the above isomorphisms with the special linear one. Shortening the notation, set
\[
\overline{\MSL}_{\eta *}^{[\star]}(Y)=\MSL_{\eta *}^{[\star]}(Y)\otimes_{\MSL_{\hphantom{[} 0}^{[2\star]}(\pt)} \BO_{\hphantom{[} 0}^{[2\star]} (\pt).
\]
By Lemma~\ref{lem_etaphi} for every small pointed motivic space $Y$ morphism $\varphi^{\SL}$ induces a natural homomorphism
\[
\overline{\varphi}_{\eta,Y}^{\SL}\colon \overline{\MSL}_{\eta *}^{[\star]}(Y) \xrightarrow{} \BO_{\eta *}^{[\star]}(Y).
\]
\begin{lemma}
\label{lem_section}
For every small pointed motivic space $Y$ there is a natural homomorphism of $\pi_{\hphantom{[} *}^{[\star]}(\pt)$-algebras
\[
t_{Y}\colon \BO_{\eta *}^{[\star]}(Y) \to \overline{\MSL}_{\eta *}^{[\star]}(Y)
\]
such that 
\begin{enumerate}
\item
$\overline{\varphi}_{\eta,Y}^{\SL}\circ t_{Y}=\id$.
\item
$t_Y(a)=1\otimes a$ for every $a\in \BO_{\eta 0}^{[2\star]}(\pt)$.
\item
$t_{\Th(E)}(\thc^{\BO}(\Tc))=\thc^{\MSL}(\Tc)\otimes 1$ for every smooth variety $X$ and every special linear bundle $\Tc=(E,\lambda)$ such that there exists a symplectic form $\phi$ on $E$ compatible with trivialization $\lambda$. 
\end{enumerate}
\end{lemma}
\begin{proof}
The following diagram commutes.
\[
\xymatrix{
\MSp_{\eta *}^{[\star]}(Y)\otimes_{\MSp_{\hphantom{[} 0}^{[2\star]}(\pt)} \BO_{\hphantom{[} 0}^{[2\star]}(\pt) \ar[r]^(0.73){\overline{\varphi}^{\Sp}_{\eta,Y}} \ar[d]_{\theta_Y} & \BO_{\eta *}^{[\star]}(Y) \\
\MSL_{\eta *}^{[\star]}(Y)\otimes_{\MSL_{\hphantom{[} 0}^{[2\star]}(\pt)} \BO_{\eta 0}^{[2\star]} (\pt) \ar[r]^(0.7){=}  &  
\overline{\MSL}_{\eta *}^{[\star]}(Y) \ar[u]_{\overline{\varphi}_{\eta,Y}^{\SL}} }
\]
Here $\theta_Y$ is induced by $\psi^{\Sp}_{Y}$. Theorem~\ref{thm_sCF} provides that $\overline{\varphi}^{\Sp}_{\eta,Y}$ is an isomorphism, thus we can take $t_{Y}=\theta_Y\circ(\overline{\varphi}^{\Sp}_{\eta,Y})^{-1}$. The first property is clear. The second property follows from the surjectivity of the natural map 
\[
\BO_{\hphantom{[} 0}^{[2\star]}(\pt)\cong \GW^{2\star}(\pt)\to \W^{2\star}(\pt) \cong \BO_{\eta 0}^{[2\star]}(\pt).
\]

For the third property recall that $\theta_Y$ and $\varphi_Y^{\Sp}$ map Thom classes of symplectic bundles to the corresponding Thom classes, so
\[
t_{\Th(E)}(\thc^{\BO}(\Tc))=t_{\Th(E)}(\thc^{\BO}(E,\phi))=\theta_Y(\thc^{\MSp}(E,\phi)\otimes 1)
=\thc^{\MSL}(E,\phi)\otimes 1=\thc^{\MSL}(\Tc)\otimes 1. 
\]
\end{proof}

Now we restrict our attention to the indices $([2\star],0)$, special linear Grassmannians $\SGr$ and corresponding Thom spaces. Recall the following "symplectic principle" for the special linear bundles (see \cite[Theorem 7]{An15}).

\begin{theorem}
\label{thm_cohsym}
Let $\Tc=(E,\lambda)$ be a special linear bundle of even rank over a smooth variety $X$. Then there exists a morphism of smooth varieties $p\colon Y\to X$ such that
$\MSL_{\eta *}^{[\star]}(Y)$ is a free $\MSL_{\eta *}^{[\star]}(X)$-module (via $p^{\MSL_\eta}$) and $p^*E$ has a canonical symplectic form $\phi$ compatible with the trivialization $p^*\lambda$.
\end{theorem}

\begin{lemma}
\label{lem_Thiso}
For the homomorphism
\[
t_{\Th(2n,m)}\colon \BO_{\eta 0}^{[2\star]}(\Th(2n,m)) \to  \overline{\MSL}_{\eta 0}^{[2\star]}(\Th(2n,m))
\]
we have 
\[
t_{\Th(2n,m)}(\thc^{\BO}(\Tc(2n,m)))=\thc^{\MSL}(\Tc(2n,m))\otimes 1.
\]
\end{lemma}
\begin{proof}
Put $\Tc(2n,m)=(E,\lambda)$. Theorem~\ref{thm_cohsym} provides a morphism of smooth varieties $p\colon Y\to \SGr(2n,m)$ such that $p^{\MSL_\eta}$ remains injective after every extension of scalars since the corresponding module stays free. Moreover, there exists a symplectic form $\phi$ on $p^*E$ agreeing with $p^*\lambda$.

Consider the following diagram.
\[
\xymatrix{
\overline{\MSL}_{\eta 0}^{[2\star]}(\Th(p^*E))  & & \BO_{\eta 0}^{[2\star]}(\Th(p^*(E))) \ar[ll]_{t_{\Th(p^*(E))}}\\
\overline{\MSL}_{\eta 0}^{[2\star]}(\Th(E)) \ar[u]^{\widetilde{p}^{\MSL_\eta}}  & & \BO_{\eta 0}^{[2\star]}(\Th(E)) \ar[u]^{\widetilde{p}^{\BO_\eta}} \ar[ll]_{t_{\Th(E)}}
}
\]
Here $\widetilde{p}\colon \Th(p^*E)\to \Th(E)$ is induced by the morphism $p$. Naturality of $t$ yields that the diagram is commutative.
Hence, by functoriality of Thom classes and Lemma~\ref{lem_section}, we obtain
\begin{multline*}
\widetilde{p}^{\MSL_\eta}t_{\Th(E)}(\thc^{\BO}(\Tc(2n,m)))=t_{\Th(p^*E)}\widetilde{p}^{\BO_\eta}(\thc^{\BO}(\Tc(2n,m)))=t_{\Th(p^*E)}(\thc^{\BO}(p^*E,p^*\lambda))=\\
=t_{\Th(p^*E)}(\thc^{\BO}(p^*E,\phi))=\thc^{\MSL_\eta}(p^*E,\phi)\otimes 1=\thc^{\MSL_\eta}(p^*\Tc(2n,m)\otimes 1)=\\
=\widetilde{p}^{\MSL_\eta}(\thc^{\MSL}(\Tc(2n,m))\otimes 1).
\end{multline*}

The claim follows since $\widetilde{p}^{\MSL_\eta}$ is injective.
\end{proof}

\begin{theorem}
\label{thm_slCF}
For every small pointed motivic space $Y$ morphism $\varphi^{\SL}$ induces an isomorphism of bigraded rings
\[
\overline{\varphi}_{\eta,Y}^{\SL}\colon \MSL_{\eta *}^{[\star]}(Y)\otimes_{\MSL_{\hphantom{[} 0}^{[2\star]}(\pt)} \BO_{\eta 0}^{[2\star]} (\pt) \xrightarrow{\simeq} \BO_{\eta *}^{[\star]}(Y).
\]
\end{theorem}
\begin{proof}
First we focus on the indices $([2\star],0)$. By Lemma~\ref{lem_section} it follows that the homomorphism
\[
\overline{\varphi}_{\eta,Y}^{\SL}\colon \overline{\MSL}_{\eta 0}^{[2\star]}(Y) \xrightarrow{} \BO_{\eta 0}^{[2\star]}(Y)
\]
is surjective. In order to check that its section $t_Y$ is surjective as well take an arbitrary element $\alpha\otimes b \in \overline{\MSL}_{\eta 0}^{[2\star]}(Y)$. One may assume that $\alpha=\beta\cup\eta^{-n}$ for some $n\in \mathbb{N}$ and $\beta \in \MSL_{\hphantom{[} -n}^{[2m-n]}(Y)\cong \MSL^{4m,2m}(\Sigma^{n,n}Y)$. By a similar argument as in Lemma~\ref{lem_smallMult} we may assume that 
\[
\Sigma_T^{2i}\Sigma^{n,n}\beta=g^{\MSL}\thc^{\MSL}(\Tc(2m+2i,(2m+2i)^2))
\]
for some $i,j\in \mathbb{N}$ and $g\in \Hom_{\Hp(k)} ((\Sigma^{n,n}Y)\wedge T^{\wedge 2i},\Th(2m+2i,(2m+2i)^2))$.
Set $r=2m+2i$ and consider the following commutative diagram.
\[
\xymatrix{
\overline{\MSL}_{\eta n}^{[2\star+2i+n]} \ar[d]_{g^{\MSL_\eta}\otimes \id} & \BO_{\eta n}^{[2\star+2i+n]} (\Th(r,r^2))  \ar[l]_(0.47){t_{\Th}} \ar[d]_{g^{\BO_\eta}}\\
\overline{\MSL}_{\eta n}^{[2\star+2i+n]} (\Sigma_T^{2i}\Sigma^{n,n}Y) & \BO_{\eta n}^{[2\star+2i+n]} (\Sigma_T^{2i}\Sigma^{n,n}Y)  \ar[l]_(0.45){t_{\Sigma^{2i}_T\Sigma^{n,n}Y}} \\
\overline{\MSL}_{\eta 0}^{[2\star]} (Y) \ar[u]^{\Sigma^{2i}_T\Sigma^{n,n}\otimes \id}_{\simeq} & \BO_{\eta 0}^{[2\star]} (Y) \ar[l]_{t_Y} \ar[u]^{\Sigma^{2i}_T\Sigma^{n,n}}_{\simeq}.
}
\]
By Lemma~\ref{lem_Thiso} and the above considerations we have
\begin{multline*}
t_Y((\Sigma^{2i}_T\Sigma^{n,n}\otimes \id)^{-1}g^{\BO_\eta}(\thc^{\BO}(\Tc(r,r^2))\cup \eta^{-n} \cup b))=\\
=(\Sigma^{2i}_T\Sigma^{n,n}\otimes \id)^{-1}(g^{\MSL_\eta}\otimes 1)t_{\Th}(\thc^{\BO}(\Tc(r,r^2))\cup\eta^{-n}\cup b)=\\
=(\Sigma^{2i}_T\Sigma^{n,n}\otimes \id)^{-1}(g^{\MSL_\eta}\otimes 1)(\thc^{\MSL}(\Tc(r,r^2))\cup\eta^{-n}\otimes b)=\\
=\beta\cup\eta^{-n}\otimes b=\alpha\cup b.
\end{multline*}

Thus the section
\[
t_Y\colon \BO_{\eta 0}^{[2\star]} (Y) \xrightarrow{}  \overline{\MSL}_{\eta 0}^{[2\star]}(Y)
\]
is surjective and $\overline{\varphi}_{\eta,Y}^{\SL}$ is an isomorphism for the indices $([2\star],0)$.

In order to obtain the claim for arbitrary indices $([n],i)$ one should first use the suspension isomorphisms 
\[
\overline{\MSL}_{\eta i}^{[n]}(Y)\cong\overline{\MSL}_{\eta 0}^{[n]}(Y\wedge S^{i,0}),\quad
\BO_{\eta i}^{[n]}(Y)\cong \BO_{\eta 0}^{[n]}(Y\wedge S^{i,0})
\]
in case of $i>0$ and the analogous ones with $\wedge S^{i,0}$ replaced by $\wedge S^{i,i}$ in case of $i<0$ and then 
\[
\overline{\MSL}_{\eta 0}^{[m]}(Y)\cong\overline{\MSL}_{\eta 0}^{[2m]}(Y\wedge S^{2m,m}),\quad
\BO_{\eta 0}^{[m]}(Y)\cong \BO_{\eta 0}^{[2m]}(Y\wedge S^{2m,m})
\]
for $m>0$ and the similar ones with $S^{2m,m}$ replaced by $S^{-2m,-m}$ for $m<0$.
\end{proof}

\begin{corollary}
\label{cor_slCFW}
For every smooth variety $X$ morphism $\varphi^{\SL}$ induces an isomorphism of bigraded rings
\[
\MSL_{\eta *}^{[\star]}(X)\otimes_{\MSL_{\hphantom{[} 0}^{[2\star]}(\pt)} \W^{2\star} (\pt) \xrightarrow{\simeq} \W^{\star}(X)[\eta,\eta^{-1}].
\]
\end{corollary}
\begin{proof}
The proof is similar to that of Corollary~\ref{cor_sCFW} with $\MSL$ replacing $\MSp$. 
\end{proof}


\begin{thebibliography}{XXXX}

\bibitem[Ad74]{Ad74}
J.~F.~Adams,
\emph{Stable homotopy and generalized homology,}
Univ. of Chicago Press, 1974.

\bibitem[An15]{An15}
A.~Ananyevskiy,
The special linear version of the projective bundle theorem,
\emph{Compositio Math.} \textbf{151} (2015), no. 3, 461--501.

\bibitem[CF66]{CF}
P.E.~Conner and E.E.~Floyd,
\emph{The relation of cobordism to K-theories,}
Lecture Notes in Mathematics, Springer-Verlag Berlin, 1966.

\bibitem[Bal99]{Bal99}
P.~Balmer,
Derived Witt groups of a scheme,
\emph{J. Pure Appl. Algebra} \textbf{141} (1999), 101--129.

\bibitem[Bal05]{Bal05}
P.~Balmer,
\emph{Witt groups}, in 
\emph{Handbook of K-theory}, Springer, Berlin, 2005, 539--576.

\bibitem[GN03]{GN}
S.~Gille, A.~Nenashev,
Pairings in triangular Witt theory,
\emph{J. Algebra} \textbf{261} (2003), 292--309.


\bibitem[Jar00]{Jar}
J.~F.~Jardine,
Motivic symmetric spectra,
\emph{Doc. Math.} \textbf{5} (2000), 445--552.

\bibitem[Mor04]{Mor04}
F.~Morel,
\emph{An introduction to $\A^1$-homotopy theory,}
Contemporary developments in algebraic K-theory, ICTP Lecture Notes, vol. XV, (Abdus Salam International Center for Theoretical Physics, Trieste, 2004), 357--441.

\bibitem[MV99]{MV}
F.~Morel, V.~Voevodsky,
$\A^1$-homotopy theory of schemes,
\emph{Publications Mathématiques de l'IH\'ES} \textbf{90} (1999), 45--143.

\bibitem[Ne07]{Ne07}
A.~Nenashev,
Gysin maps in Balmer-Witt theory,
\emph{J. Pure Appl. Algebra} \textbf{211} (2007), 203--221.

\bibitem[PS03]{PS}
I.~Panin (after I. Panin and A. Smirnov),
Oriented cohomology theories of algebraic varieties,
\emph{K-Theory} \textbf{30} (2003), 265--314.

\bibitem[PW10a]{PW1}
I.~Panin and C.~Walter,
Quaternionic Grassmannians and Pontryagin classes in algebraic geometry,
\emph{arXiv:1011.0649} (2010).

\bibitem[PW10b]{PW2}
I.~Panin and C.~Walter,
On the motivic commutative spectrum BO,
\emph{arXiv:1011.0650} (2010).

\bibitem[PW10c]{PW3}
I.~Panin and C.~Walter,
On the algebraic cobordism spectra MSL and MSp,
\emph{arXiv:1011.0651} (2010).

\bibitem[PW10d]{PW4}
I.~Panin and C.~Walter,
On the relation of the symplectic algebraic cobordism to hermitian K-theory,
\emph{arXiv:1011.0652} (2010).

\bibitem[Sch10a]{Sch10a}
M.~Schlichting,
Hermitian K-theory of exact categories,
\emph{J. K-theory} \textbf{5} (2010), no. 1, 105--165.

\bibitem[Sch10b]{Sch10b}
M.~Schlichting,
The Mayer-Vietoris principle for Grothendieck-Witt groups of schemes,
\emph{Invent. Math.}, \textbf{179} (2010), 349--433.


\bibitem[Sch12]{Sch12}
M.~Schlichting,
Hermitian K-theory, derived equivalences and Karoubi's Fundamental Theorem,
\emph{arXiv:1209.0848} (2012).

\bibitem[Voe98]{V}
V.~Voevodsky,
$\A^1$-homotopy theory,
\emph{Doc. Math.}, Extra Vol. \textbf{I} (1998), 579--604.

\end{thebibliography}
\end{document}